\documentclass[11pt]{article}

\usepackage{hyperref,amsmath,amssymb,amscd,amsthm,makeidx,txfonts,graphicx,url,bm}
\usepackage{a4wide,xy,tablefootnote}
\usepackage{float}
\usepackage{tikz}
\usetikzlibrary{fit}
\let\oldmarginpar\marginpar
\renewcommand\marginpar[1]{\-\oldmarginpar[\raggedleft\footnotesize #1]
{\raggedright\footnotesize #1}}

\makeindex

\xyoption{all}
\input{xypic}

\numberwithin{equation}{section}

\newtheorem{theorem}{Theorem}[section]
\newtheorem{proposition}[theorem]{Proposition}
\newtheorem{corollary}[theorem]{Corollary}

\newtheorem{lemma}[theorem]{Lemma}

\theoremstyle{remark}
\newtheorem{remark}[theorem]{Remark}
\newtheorem{example}[theorem]{Example}

\theoremstyle{definition}
\newtheorem{definition}[theorem]{Definition}

\def\bes{\begin{eqnarray*}}
\def\ees{\end{eqnarray*}}

\def\bm{{\bm a}}

\def\m{{\mathfrak m}}
\def\F{{\rm F}}

 \DeclareMathOperator{\Hom}{Hom}
\DeclareMathOperator{\Tr}{Tr}

\def\V{\mathbb{V}}
\def\Q{\rm{O}}
\def\C{\mathbb{C}}
\def\M{{\rm{M}}}

\def\calR{{\mathcal{R}}}

\def\calE{{\mathcal{E}}}
\def\calF{{\mathcal{F}}}

\def\calH{\mathcal{H}}
\def\rp{{\rm p}}
\def\K{{\rm K}}
\def\H{{\rm H}}
\def\bm

\def\1{{\bf 1}}

\def\k{\mathbf{k}}

\def\I{{\rm I}}

\def\N{\mathbb{Z}_{\geq 0}}

\def\F{\mathbb{F}}

\def\Q{\mathbb{Q}}
\def\D{{\rm D}}
\def\T{{\rm{T}}}

\def\Z{\mathbb{Z}}

\def\gl{{\mathfrak g\mathfrak l}}

\newcommand{\nc}{\newcommand}
\nc{\bM}{{\mathbb M}}
\nc{\Sp}{{\rm Sp}}

\def\rep{{\rm Rep}}
\def\End{{\rm End}}
\def\k{{\rm k}}
\nc{\R}{{\rm R}} 
\def\Alg{{\bf Alg}}
\def\Gr{{\bf Gr}}

\nc{\Rt}{{\rm R}_2}

\nc{\op}[1]{\mathop{\mathchoice{\mbox{\rm #1}}{\mbox{\rm #1}}
{\mbox{\rm \scriptsize #1}}{\mbox{\rm \tiny #1}}}\nolimits}
\nc{\al}{\alpha}

\nc{\ep}{\varepsilon} \nc{\ga}{\gamma} \nc{\Ga}{Q}
\nc{\la}{\lambda} \nc{\La}{\Lambda} \nc{\si}{\sigma}
\nc{\Sig}{{Q}} \nc{\Om}{\Omega} \nc{\om}{\omega}
\nc{\coker}{{\rm coker}}

\nc{\SL}{{\rm SL}} \nc{\GL}{{\rm GL}} \nc{\PGL}{{\rm PGL}}
\nc{\G}{{\rm G}}
\nc{\bV}{{\mathbb V}}
\nc{\g}{{\mathfrak g}}
\nc{\card}{{\rm card}}

\nc{\uni}{{\rm uni}}

\nc{\beq}[1]{\begin{eqnarray}\label{#1}}
\nc{\eeq}{\end{eqnarray}}

\def\U{{\rm U}}

\nc{\cpt}{{\op{cpt}}} \nc{\Dol}{{\op{Dol}}} \nc{\DR}{{\op{DR}}}
\nc{\B}{{\op{B}}} \nc{\Triv}{\op{Triv}} \nc{\Hod}{{\op{Hod}}}
\nc{\Log}{{\op{Log}}} \nc{\Exp}{{\op{Exp}}} \nc{\Est}{E_{\op{st}}}
\nc{\Hst}{H_{\op{st}}} \nc{\Left}[1]{\hbox{$\left#1\vbox to
  10.5pt{}\right.\nulldelimiterspace=0pt \mathsurround=0pt$}}
\nc{\kight}[1]{\hbox{$\left.\vbox to
  10.5pt{}\right#1\nulldelimiterspace=0pt \mathsurround=0pt$}}
\nc{\LEFT}[1]{\hbox{$\left#1\vbox to
  15.5pt{}\right.\nulldelimiterspace=0pt \mathsurround=0pt$}}
\nc{\kIGHT}[1]{\hbox{$\left.\vbox to
  15.5pt{}\right#1\nulldelimiterspace=0pt \mathsurround=0pt$}}

\nc{\bee}{{\bf E}} \nc{\bphi}{{\bf \Phi}}

\newcommand\addvmargin[1]{
  \node[fit=(current bounding box),inner ysep=#1,inner xsep=0]{};
}

\begin{document}

\title{Locally free representations of quivers over commutative Frobenius algebras}

\author{Tam\'as Hausel
\\ {\it IST Austria}
\\{\tt tamas.hausel@ist.ac.at} \and Emmanuel Letellier \\ {\it
 Universit\'e de Paris, IMJ-PRG, CNRS}  \\ {\tt
  emmanuel.letellier@imj-prg.fr} \and Fernando Rodriguez Villegas
\\ 
{\it  ICTP Trieste} \\ {\tt villegas@ictp.it}  }

\pagestyle{myheadings}

\maketitle

\begin{abstract} In this paper we investigate locally free representations of a quiver $Q$ over a commutative Frobenius algebra $\R$ by arithmetic Fourier transform. When the base field is finite we prove that the number of isomorphism classes of absolutely indecomposable locally free representations of fixed rank is independent of the orientation of $Q$. We also prove that the number of isomorphism classes of locally free  absolutely indecomposable representations of the preprojective algebra of $Q$ over $\R$ equals the number of isomorphism classes of locally free absolutely indecomposable representations of $Q$ over $\R[t]/(t^2)$. Using these results together with results of Geiss, Leclerc and Schr\"oer we give, when $\k$ is algebraically closed, a  classification of  pairs $(Q,\R)$ such that the set of isomorphism classes of  indecomposable locally free representations of $Q$ over $\R$ is finite. Finally when the representation is free of rank $1$ at each vertex of $Q$, we study the function that counts the number of isomorphism classes of absolutely indecomposable locally free representations of $Q$  over the Frobenius algebra $\F_q[t]/(t^r)$. We prove that they are polynomial in $q$ and their generating function is rational and satisfies a functional equation.

\end{abstract}

\tableofcontents

\section{Introduction}\label{intro}

 Let $\R$ be a commutative Frobenius  $\k$-algebra over a field $\k$ (which will be either finite or algebraically closed).  This means that $\R$ is finite dimensional as a $\k$-vector space and equipped with a Frobenius $1$-form $\lambda:\R\to \k$, i.e. one which is non-zero on any nontrivial ideal of $\R$. Examples to keep in mind are the truncated polynomial rings $\k_d:=\k[t]/(t^d)$, in particular $\k_2=\k[\epsilon]/(\epsilon^2)$ which we usually abbreviate to $\k[\epsilon]$ the algebra of dual numbers. More generally if $\R$ is a Frobenius $\k$-algebra so is $\R[\epsilon]:=\R[\epsilon]/(\epsilon^2)$ (see \S \ref{Frobsection} for more details). 

We fix $Q=(I,E)$ a finite quiver, with set of vertices $I$ and set of arrows $E$. A representation of $Q$ over the ring $\R$ is given by an $\R$-module $M_i$ at each vertex $i\in I$ together with a $\R$-module homomorphism $M_a:M_i\rightarrow M_j$ for any arrow $a:i\rightarrow j\in E$. The category of representations of $Q$ over $\R$ is equivalent to the category $\R Q$-Mod where $\R Q$ denotes the path algebra of $Q$. 
According to \cite[Definition 1.1]{GLS}, a representation $\{M_i,M_a\}_{i\in I,a\in E}$ of $Q$ is said to be \emph{locally free} if the $\R$-modules $M_i$ are all free. If we denote by $e_i\in\R Q$ the idempotent corresponding to vertex $i\in I$, locally free representations of $Q$  correspond to $\R Q$-modules $M$ such that for all $i\in I$, the $\R$-module $M_i=e_iM$ is free.

When $\R=\k$ is a field Gabriel \cite{gabriel1} in 1972 followed by  Kac in \cite{kac} in 1982 initiated the detailed study of the representation category $\k Q$-Mod. They discovered a deep connection with the representation theory of Kac-Moody algebras with symmetric Cartan matrices. To find the analogue connection with Kac-Moody algebras with symmetrizable but non-symmetric Cartan matrices, Gabriel \cite{gabriel2} in 1974 followed by Dlab and Ringel \cite{dlab-ringel} in 1974 introduced the notion of modulated graphs, in particular allowed various finite field extensions attached to the vertices of the graph. To get a uniform theory working over any, for example algebraically closed, fields Geiss--Leclerc--Schr\"oer \cite{GLS} in 2015  introduced the study of representations of quivers over $\R=\k_d$ the truncated polynomial ring over a field $\k$. Here they managed to find the desired connection with non-simply laced Kac-Moody algebras. Later in 2016 Geuenich \cite{geuenich} extended results of \cite{GLS} to the case when $\R$ is a Frobenius algebra. Finally other recent papers \cite{ringel-zhang1, ringel-zhang2, li-ye, kaplan} study different approaches to representations of quivers over various Frobenius algebras.

Many of the  results mentioned above have been achieved by use of reflection functors. In
this paper we will concentrate on the subcategory of locally free
representations. It is not stable under reflection functors, which use
kernels and cokernels. Thus
usual techniques will not be available. Instead we will use arithmetic
harmonic analysis and prove results showing that locally free
representations behave better when changing orientations of the
quiver, and have interesting connections to representations of
preprojective algebras.

Let $\alpha\in \N^I$ (elements of $\N^I$ will be called \emph{rank vectors}) and put
$$\rep^\alpha(Q,\R):=\bigoplus_{i\rightarrow j\in E} 
\Hom_\R(\R^{\alpha(i)},\R^{\alpha(j)})$$ the $\R$-module of locally free $\R$-representations of the
quiver $Q$ of rank $\alpha$. The group $$\G_\alpha(\R):=\varprod_{i\in I} \GL(\R^{\alpha(i)})$$ acts on $\rep^\alpha(Q,\R)$ and so gives a homomorphism \beq{action} \rho:\G_\alpha(\R)\to {\rm Aut}_\R(\rep^\alpha(Q,\R)).\eeq Finally let $\g_\alpha(\R)=\oplus_{i\in} \gl(\R^{\alpha(i)})$ be the Lie algebra of $\G_\alpha(\R)$ and $\varrho:\g_\alpha(\R)\to \End(\rep^\alpha(Q,\R))$ the derivative of \eqref{action}, the infinitesimal action. 
The set of isomorphism classes of locally free $\R$-representations of $Q$ of dimension $\alpha$, 
can be identified with the orbit space $\rep^\alpha(Q,\R)/\G_\alpha(\R)$. 

If $\k$ is a finite field we put

\bes m_\alpha(\R Q):=\card(\rep^\alpha(Q,\R)/\G_\alpha(\R))\ees  the number of isomorphism classes of locally free $\R$-representations of $Q$ of dimension $\alpha$. 

If $\k$ is algebraically closed, the number of isomorphism classes could be infinite and so instead we consider  the positive integers (cf.  \S \ref{algclose} for more details)

$$
d:={\rm dim}\,\{(x,g)\in\rep^\alpha(Q,\R)\times\G_\alpha(\R)\,|\,g\cdot x=x\}-{\rm dim}\, \G_\alpha(\R)
$$
and  the number $c$ of irreducible components of $\{(x,g)\in\rep^\alpha(Q,\R)\times\G_\alpha(\R)\,|\,g\cdot x=x\}$ of maximal dimension (here our $\R$-schemes are regarded as $\k$-schemes). The pair $(c,d)$ gives us thus an estimation of the size of the orbit space $\rep^\alpha(Q,\R)/\G_\alpha(\R)$. Note that if $d=0$, then the orbit space is finite of cardinality $c$.
By abuse of notation we also use the notation $m_\alpha(\R Q)$ to denote the pair $(c,d)$ in the algebraically closed field case. 

Given two representations of $Q$ over $\R$ we can form their direct sum. We define an {indecomposable} representation of $Q$ over $\R$ as one  which  is not isomorphic with a direct sum of two non-trivial representations. If we assume that $\R$ is local then the category of locally free representations is Krull-Schmidt. Therefore when talking about indecomposable representations we will always assume  that $\R$ is a local Frobenius $\k$-algebra which we will further assume to be \emph{split} (see \S \ref{Frobsection}). When $M$ is a representation of $Q$ over $\R$ then it is \emph{absolutely indecomposable} if the representation given by $${M}_a\otimes_\R \overline{\R}: M_i\otimes_\R \overline{\R}\to M_j\otimes_\R \overline{\R}$$ is indecomposable over $\overline{\R}:=\R\otimes_\k \overline{\k},$ where $\overline{\k}$ is the algebraic closure of $\k$. Let us denote by $\rep^\alpha_{\rm a.i.}(Q,\R)$  the subset of $\rep^\alpha(Q,\R)$ of absolutely indecomposable representations. 

If $\k$ is a finite field we put

$$
a_\alpha(\R Q):=\card(\rep^\alpha_{\rm a.i.}(Q,\R)/\G_\alpha(\R)).
$$Otherwise $\k$ is algebraically closed, absolutely indecomposable just means indecomposable  and we define $a_\alpha(\R Q)$ as the pair $(c,d)$ where $c$ is the number of irreducible components of $\{(x,g)\in\rep^\alpha_{\rm a.i.}(Q,\R)\times\G_\alpha(\R)\,|\,g\cdot x=x\}$ of maximal dimension and $d$ is

$${\rm dim}\,\{(x,g)\in\rep^\alpha_{\rm a.i.}(Q,\R)\times\G_\alpha(\R)\,|\,g\cdot x=x\}-{\rm dim}\, \G_\alpha(\R).$$
Denote by $|\alpha|$ the integer $\sum_{i\in I}\alpha(i)$.

\begin{theorem}\label{main1} Assume that $\R$ is a  Frobenius algebra over a finite or algebraically closed field $\k$. Let $Q$ be a quiver and let $Q'$ be another quiver obtained from $Q$ by changing the orientation of some of the arrows. Then $$m_\alpha(\R Q)=m_\alpha(\R Q').$$If moreover $\R$ is local, split and if $\k$ contains a primitive $|\alpha|$-th root of unity then $$a_\alpha(\R Q)=a_\alpha(\R Q').$$
\label{indeptheo}\end{theorem}

%\begin{remark}Theorem \ref{indeptheo} is not true anymore if one considers the full representation theory instead of locally free representations (see Remark \ref{counterexfull} for details).\end{remark}

The version of this theorem when $\R=\k$ was called the { \em fundamental lemma} in \cite{kraft-riedtmann} in Kac's characterization \cite{kac} of roots  in terms of $a_\alpha(\k Q).$ 

To formulate our second main result consider the double quiver $\overline{Q}=(I,E\cup E^*)$ obtained from $Q$ by adjoining a new arrow $a^*:j\leftarrow i$ for each arrow $a:i\rightarrow j\in E$, and denote by $Q^*$ the \emph{opposite} quiver $(I,E^*)$ of $Q$.  
Then define the moment map
\beq{moment}\mu_\R=\mu_{\R,\alpha}: \begin{array}{ccc}\rep^\alpha(\overline{Q},\R) &\to& \g_\alpha(\R)\\ M&\mapsto& \sum_{a\in E} M_a M_{a^*}- M_{a^*}M_a \end{array} \eeq 

The group $\G_\alpha(\R)$  acts naturally on both sides, and the map $\mu_\R$ is equivariant. Now denote by $\R\overline{Q}$  the path algebra of $\overline{Q}$ and let $\sigma=\sum_{a\in E} aa^*-a^*a\in \R\overline{Q}$. Then the points of $\mu_\R^{-1}(0)$ correspond to the locally free representations of the $\R$-algebra $$\Pi_\R(Q)=\R\overline{Q}/(\sigma)$$ of dimension $\alpha$. The $\R$-algebra $\Pi_\R(Q)$ is called the preprojective algebra of $Q$ over $\R$.  So $\mu_\R^{-1}(0)/\G_\alpha(\R)$ can be thought as the set of isomorphism classes of locally free representations of $\Pi_\R(Q)$. We define $m_\alpha(\Pi_\R(Q))$ as we defined $m_\alpha(\R Q)$ with $\rep^\alpha(Q,\R)$ replaced by $\mu^{-1}_\R(0)$. We also define $a_\alpha(\Pi_\R(Q))$ similarly as $a_\alpha(\R Q)$ with $\rep^\alpha_{\rm a.i.}(Q,\R)$ replaced by the set of absolutely indecomposable  locally free representations of $\Pi_\R(Q)$. 
Then we have

\begin{theorem}\label{main2} If $\R$ is a Frobenius algebra over a finite or algebraically closed field $\k$ then we have 
 $$m_\alpha(\Pi_\R(Q))=m_\alpha(\R[\epsilon] Q).$$If moreover $\R$ is local, split and if $\k$ contains a primitive $|\alpha|$-th root of unity then $$a_\alpha(\Pi_\R(Q))=a_\alpha(\R[\epsilon] Q).$$
	\label{maintheoquiv}\end{theorem}

%\begin{remark}In \S \ref{counterex} we give counter-examples to Theorem \ref{maintheoquiv} when $R$ is not a Frobenius algebra \end{remark}

When $\R=\k$ is a field this result relates representations of a quiver over $
\k_2=\k[\epsilon]$ and its preprojective algebra. Thus it unexpectedly  bridges two papers \cite{GLS1,GLS} of Geiss--Leclerc--Schr\"oer. Also in the case of the Jordan quiver and $\R=\k$ a finite field the first part of Theorem~\ref{main2} is \cite[Theorem 1]{jambor-plesken} of Jambor--Plesken, which was the starting point of our research. In this case Theorem 1.2 amounts to the agreement of the number of similarity classes of $n\times n$ matrices over $\k_2$ with the number of isomorphism classes of commuting pairs of $n\times n$ matrices over $\k$.   We will discuss the case of the Jordan quiver in more detail in a forthcoming publication. 

We start in Section 2 by introducing Frobenius algebras. In Section 3 we discuss the Fourier transform for finitely generated modules over a finite Frobenius algebra. In Section 4  we apply such Fourier transforms for representations of algebraic group schemes, culminating in Proposition~\ref{mainprop} and its many corollaries among them formula \eqref{mainfor}. In Section 4.3 we give  counterexamples to \eqref{mainfor}, when $\R$ is not a Frobenius algebra. In Section 5 we develop the technical machinery to deduce our main Theorems~\ref{main1} and \ref{main2} from the finite field versions. In Section 6 we classify all quivers and Frobenius algebras  with finitely many locally free indecomposable representations. Finally in Section 7 we compute the number of indecomposable representations of a quiver of rank vector $\bf 1$ over $\k_d$, prove that their generating function is rational and satisfies an Igusa-type functional equation.

\paragraph{Acknowledgements}  Special thanks go to Christof Geiss, Bernard Leclerc and Jan Schr\"oer for explaining their work but also for sharing some unpublished results with us. We also thank the referee for many useful suggestions. We would like to thank Tommaso Scognamiglio for pointing out a mistake in the proof of Proposition 5.17 in an  earlier version of the paper. We would like also to thank Alexander Beilinson, Bill Crawley-Boevey, Joel Kamnitzer, and Peng Shan for useful discussions. 

The work on this paper started at a Research in Pairs retreat at the Oberwolfach Forschunginstitute in 2014. We are
grateful for the Institute for the ideal conditions. Research of T.H. was supported by the  Advanced Grant ``Arithmetic and physics of Higgs moduli spaces'' no. 320593 of the European Research Council, by grants no. 144119, no. 153627
and NCCR SwissMAP, funded by the Swiss National Science Foundation. Research of E.L. was supported by ANR-13-BS01-0001-01. 

\section{Frobenius algebras}\label{Frobsection}

Let $\k$ be field, which in this paper will be either finite or algebraically closed. Let $\R$ be 
a  commutative $\k$-algebra which is  finite dimensional as a $\k$ vector space.   We say that $\R$ is a Frobenius
algebra if we have an $\R$-module isomorphism \beq{dual}\R\cong_{\R}
\R^{\! *},\eeq where $\R^*:=\Hom_\k(\R,\k)$ is an $\R$-module by $r\lambda(x)=\lambda(rx)$ for $\lambda\in\R^*$, $r,x\in \R$.

Equivalently to the condition \eqref{dual} a Frobenius algebra is defined by the 
$\k$-linear functional $\lambda:\R\to \k$, called {\em Frobenius form}.
It corresponds to $1\in\R$ in the isomorphism of \eqref{dual}. For \eqref{dual} to be an isomorphism  is equivalent with $\lambda$   not vanishing on any
non-zero ideal of $\R$. 

Every Frobenius algebra $\R$ can be decomposed as a direct product of local Frobenius algebras $\R\cong \R_1\times\cdots \times \R_n$. Local Frobenius algebras are precisely the finite dimensional local algebras with unique minimal ideal.   For a local Frobenius algebra $\R$ we denote by $\mathfrak{m}\triangleleft \R $ the maximal ideal and by  $\k^\prime=\R/\mathfrak{m}$ the residue field. When $\k\cong \k^\prime$ we call the local Frobenius algebra $\R$ { \em split}. We know from \cite[Theorem 7.7]{eisenbud} that a finite dimensional local algebra always contains its residue field. Thus we can consider any local Frobenius algebra $\R$ as a split  local Frobenius $\k^\prime$-algebra.

\begin{example}Let $\R$ be a finite dimensional $\k$-algebra. The trivial extension of $R$, i.e. the algebra
	$$\T^*\R:=\R^*\times \R$$ given by the multiplication

	$$(q_1,a_1)\cdot(q_2,a_2)=(a_2q_1+a_1q_2,a_1a_2)$$ is always a Frobenius $\k$-algebra with Frobenius form $\lambda(q,a)=q(1)$. Indeed if $(q,a)$ such that $a\neq 0$ we can find $q^\prime$ such that $q^\prime(a)\neq 0$ and so $$\lambda((q,a)\cdot (q^\prime,0))=\lambda(aq^\prime,0)=q^\prime(a)\neq 0.$$ Similarly if $q\neq 0$ there is an $a^\prime$ such that $q(a^\prime)\neq 0$ and so $\lambda((q,a)\cdot(0,a^\prime))=q(a^\prime)\neq 0$. Thus $\lambda$ is non-zero on any proper ideal in $\T^* \R$.
	\label{TR}\end{example}

\begin{example} Let $\R$ be a Frobenius algebra over a field $\k$ with Frobenius form $\lambda:\R\to \k$. Consider the finitely generated $\k$-algebra
	$\R_d:=\R[t]/(t^d)$ with $\lambda'\in \Hom_\k(\R_d,\k)$ defined by $$\lambda'(a_0+a_1t+\dots+a_{d-1}t^{d-1})=\lambda(a_{d-1}).$$ As every non-zero ideal $I\triangleleft \R$ intersects
	$t^{d-1}\R$ non-trivially, $\lambda'$ is nonzero  on $I$. Thus $\R_d$ is a Frobenius $\k$-algebra.
	
When $d=2$ we will denote $\R_2=\R[\epsilon]/(\epsilon^2)$ by $\R[\epsilon]$.
	When $\R$ is a Frobenius algebra, we have an isomorphism $\R[\epsilon]\cong \T^*\R$.  %Examples of local
	%commutative Frobenius $\k$-algebras to keep in mind are finite  field extensions $\k \subset {\rm K}$ and
	%more generally
	%the $t$-adic rings ${\rm K} [t]/(t^d)$. These are the only examples of local
	%commutative Frobenius $\k$-algebras which are principal ideal rings. 
\end{example}

\section{Fourier transform}\label{Fouriergeneral}

Let $\M$ be a finitely generated module over a Frobenius $\k$-algebra $\R$. We define the dual $\R$-module as

$$
\M^\vee:={\rm Hom}_\R(\M,\R).
$$
The $\R$-module $\M$  is then \emph{reflexive}, namely the canonical map $\M\rightarrow(\M^\vee)^\vee$ is an isomorphism by \cite[4.12.21(a)]{hazewinkel} as $\R$ is a \emph{Frobenius ring}. As this is crucial for us we will prove the reflexivity property in the case where $\k$ is a finite field.

Define the \emph{Pontryagin dual} $\R$-module of $\M$ as

$$
\M^\wedge:={\rm Hom}_{\text{Groups}}(\M,\C^\times).
$$

From now on in this section we assume that $\k$ is a finite field and $\R$ a Frobenius $\k$-algebra with $1$-form $\lambda:\R\rightarrow\k$, and that $\M$ is a finitely generated $\R$-module. In particular $M$ is a finite set which will be convenient to do Fourier analysis. We choose once and for all a non-trivial additive character $\psi:\k\rightarrow\C^\times$ and we put $\varpi:=\psi\circ\lambda:\R\rightarrow\C^\times$. 

Note that the kernel of the character $\varpi$ does not contain non-trivial ideals of $\R$ and so we have the following result.

\begin{lemma} The additive character $\varpi$ is a generator of the $\R$-module $\R^\wedge$.
\end{lemma}

\begin{proposition} The map $$\M^\vee\rightarrow \M^\wedge,\,\,\,f\mapsto \varpi\circ f$$ is an isomorphism of $\R$-modules.
\label{Pont}\end{proposition}

\begin{proof} Let $g\in \M^\wedge$. Consider $\M$ as the quotient of some $\R^r$ and denote by $\pi$ the quotient map $\R^r\rightarrow \M$. Let $(x_1,\dots,x_r)$ be the canonical basis of $\R^r$. Then for each $i=1,\dots,r$,  let $\alpha_i\in\R^\wedge$ be defined by $\alpha_i(h)=(g\circ\pi)(h x_i)$ for $h\in\R$. Since $\varpi$ is a generating character of $\R^\wedge$, we get for each $i=1,\dots,r$, an element $\lambda_i\in\R$ such that $\alpha_i=\lambda_i\cdot\varpi$. 

We let $\overline{f}:\R^r\rightarrow\R$ be the $\R$-linear form defined by $\overline{f}(x_i)=\lambda_i$ for all $i=1,\dots,r$. We then have $g\circ\pi=\varpi\circ\overline{f}$. We now prove that there exists a unique $\R$-linear form $f:\M\rightarrow\R$ making the following diagram commutative

$$
\xymatrix{\R^r\ar[r]^-\pi\ar[rd]_{\overline{f}}&\M\ar[r]^-g\ar[d]^f&\C^\times\\
&\R\ar[ru]_\varpi&}
$$
We need to see that ${\rm Ker}(\pi)\subset{\rm Ker}(\overline{f})$. For $x\in{\rm Ker}(\pi)$, we have $(\varpi\circ\overline{f})(x)=(g\circ\pi)(x)=1$ and so $\overline{f}(x)\in{\rm Ker}(\varpi)$. Moreover for all $\lambda\in\R$, we have $\lambda\overline{f}(x)=\overline{f}(\lambda x)$ and $\lambda x\in{\rm Ker}(\pi)$, and so ${\rm Ker}(\varpi)$ contains the ideal generated by $\overline{f}(x)$. Since the kernel of $\varpi$ does not contain non-trivial ideals of $\R$, we must have $\overline{f}(x)=0$ for all $x\in{\rm Ker}(\pi)$. 
\end{proof}

\begin{corollary}The $\R$-module $\M$ is reflexive.
\label{reflexive}\end{corollary}

We thus have a  natural perfect pairing

$$
\langle\,,\,\rangle:\M\times\M^\vee\rightarrow\R.
$$

For a finite set $A$, denote by $\C[A]$ the $\C$-vector space of all functions $A\rightarrow \C$. It is endowed with the inner product

$$
(f,g)_A:=\sum_{x\in A}f(x)\overline{g(x)},
$$
for all $f,g\in\C[A]$.

Let $X$ be a finite non-empty set. Consider the correspondence 

$$
\xymatrix{&X\times\M\times \M^\vee\ar[dl]_{p_{12}}\ar[dr]^{p_{13}}&\\
X\times \M&&X\times\M^\vee}
$$
and define the Fourier transform $\mathcal{F}:\C[X\times\M]\rightarrow\C[X\times\M^\vee]$ by 

$$
\mathcal{F}(f):=(p_{13})_!\left((p_{12})^*(f)\otimes (p_{23})^*(\varpi\langle,\rangle)\right),
$$
for all $f\in\C[X\times \M]$. Namely

$$
\mathcal{F}(f)(x,v)=\sum_{u\in\M}\varpi(\langle u,v\rangle)f(x,u),
$$
for all $(x,v)\in X\times\M^\vee$.

We have the following standard proposition.

\begin{proposition}(1) The rescaled Fourier transform $\card(\M)^{-\frac{1}{2}}\,\mathcal{F}$ is an isometry with respect to $\left(\,,\,\right)_{X\times \M}$ and $\left(\,,\,\right)_{X\times \M^\vee}$.

(2) For $f,g\in\C[X\times \M]$, we have 

$$
\calF(f*g)=\calF(f)\,\calF(g),
$$
where the convolution $f*g$ of $f$ by $g$ is the function $(x,u)\mapsto
(f*g)(x,u)=\sum_{u_1+u_2=u}f(x,u_1)g(x,u_2)$.
\label{propFour}
\end{proposition}

\begin{proof}Indeed 

\begin{align*}
\left(\mathcal{F}(f),\mathcal{F}(g)\right)_{X\times\M^\vee}&=\sum_{(x,v)\in X\times\M^\vee}\left(\sum_{u\in\M}\varpi(\langle u,v\rangle)f(x,u)\right)\left(\sum_{w\in\M}\varpi(\langle -w,v\rangle)\overline{g(x,w)}\right)\\
&=\sum_{x,u,w}f(x,u)\overline{g(x,w})\sum_v\varpi(\langle u-w,v\rangle).
\end{align*}

We need to see that for any $u\in\M$, we have

$$
\sum_{v\in\M^\vee}\varpi\big(\langle u,v\rangle\big)=\begin{cases}|\M|&\text{ if } u=0,\\0&\text{ otherwise}.\end{cases}
$$
The map $\M^\vee\rightarrow\C^\times$, $v\mapsto\varpi(\langle u,v\rangle)$ is an additive character. Since the kernel of $\varpi$ does not contain non-zero ideals of $\R$, this map is non-trivial as long as there exists an element $v$ of $\M^\vee$ such that $\langle u,v\rangle\neq 0$. Since the pairing $\langle\,,\,\rangle$ is non-degenerate such an element $v$ exists when $u\neq 0$.

For (2) we have

\begin{align*}
\calF(f*g)(x,v)&=\sum_{u\in M}\varpi(\langle u,v\rangle)(f*g)(x,u)\\
&=\sum_{u\in M}\varpi(\langle u,v\rangle)\sum_{u_1+u_2=u}f(x,u_1)g(x,u_2)\\
&=\sum_{u_1,u_2\in M}\varpi(\langle u_1+u_2,v\rangle)f(x,u_1)g(x,u_2)\\
&=\sum_{u_1\in M}\varpi(\langle u_1,v\rangle)f(x,u_1)\sum_{u_2\in M}\varpi(\langle u_2,v\rangle)g(x,u_2)\\
&=\calF(f)(x,v)\calF(g)(x,v).
\end{align*}
\end{proof}

We note that if a finite group $\G$ acts on $X$ and on $\M$ via $\R$-automorphisms, then the induced $\G$-action on $\M^\vee$ will make the pairing $\langle\,,\,\rangle$ equivariant and so the Fourier transform $\calF$ is a $\G$-equivariant isomorphism. Since the dimension of the $\C$-vector space of $\G$-invariant functions agrees with the number of $\G$-orbits, we  deduce  the following lemma.

\begin{lemma} \label{brauer} Let $M$ be a finitely generated module over a finite Frobenius algebra $\R$. Let a finite group $\G$ act on $M$ by via $\R$-automorphisms, and also operate on a finite set $X$.  Then we have $$\card((X\times\M)/\G)=\card((X\times \M^\vee)/\G),$$ i.e. the number of $\G$-orbits on $X\times \M$ agree with the number of $\G$-orbits on $X\times \M^\vee$. 
\end{lemma}

 \begin{remark} When $\R$ is a finite field and $X$ is a point, this observation is  attributed to Brauer \cite[Remark 5.5]{kraft-riedtmann}.  Also the orbit structure could be very different on the two sides. \end{remark}

\section{Fourier transform for group representations}

\subsection{Moment maps associated to group scheme actions}

Let $\k$ be a field. We denote by $\Alg_\k$ and $\Gr$ the categories of $\k$-algebras and groups. Let $\G$ be an algebraic affine $\k$-group scheme. By notation abuse we will use the same letter $\G$ to denote the associated $\k$-group functor $\Alg_\k\rightarrow \Gr$, $A\mapsto \G(\R)={\rm Hom}_{\k-{\bf alg}}(\k[\G],\R)$. 

Recall that, as a $\k$-group functor,  the Lie algebra $\g$ of $\G$ is defined by $\g(\R):={\rm Ker}(\G(\rp_\R))$ where for any $\k$-algebra $\R$ we put $\rp_\R:\R[\epsilon]\rightarrow \R$, $a+\epsilon b\mapsto a$. Denote by $e:\k[\G]\rightarrow\k$ the counit (the identity element in $\G(\k)$) and put $\I={\rm Ker}(e)$. Then for any $\k$-algebra $\R$, we have \cite[\S 10.19]{milne}

$$
\g(\R)\cong{\rm Hom}_\k(\I/\I^2,\R). 
$$  
From this we see that $\g(\R)$ is an Abelian group in fact an $\R$-module. We have  that the canonical morphism 

\beq{lie1}
\g(\k)\otimes_\k\R={\rm Hom}_\k(\I/\I^2,\k)\otimes_\k\R\rightarrow {\rm Hom}_\k(\I/\I^2,\R)=\g(\R)
\eeq
is an isomorphism \cite[\S 2, Proposition 2]{Bour}.
\bigskip

By definition $\g(\R)$ is a  normal subgroup of $\G(\R[\epsilon])$ and so the latter acts on $\g(\R)$ by conjugation. As $\g(\R)$ is Abelian this action factors through an action of $\G(\R)=\G(\R[\epsilon])/\g(\R)$ on $\g(\R)$. This is the so-called adjoint action of $\G$ on its Lie algebra, denoted \beq{adjoint}{\rm Ad}:\G(\R)\to {\rm Aut}(\g(\R)).\eeq Using the section $i_\R:\R\rightarrow\R[\epsilon]$, $a\mapsto a$ of $\rp_\R$ we see that the exact sequence $$0\rightarrow \g(\R)\rightarrow\G(\R[\epsilon])\rightarrow \G(\R)\rightarrow 1$$ splits and so

$$
\G(\R[\epsilon])\simeq \g(\R)\rtimes\G(\R).
$$

\bigskip

From now on we assume that $\V$ is a  finite-dimensional $\k$-vector space which we regard as a $\k$-group functor 

$$
\V(\R):=\V\otimes_\k\R,
$$
for any $\k$-algebra $\R$.
%\todo{Tamas: so we think of $\V$ also as a $\k$-group functor?}

Denote by $\GL(\V)$ the $\k$-group scheme with $\k$-group functor  $$\R\mapsto {\rm End}_\R(\V(\R))^\times.$$

\begin{remark} Here we collect some simple observations how we can work explicitly with $\GL(\V)(\R[\epsilon])$.
	
\noindent (i)	 We can write the elements of $\GL(\V)(\R[\epsilon])$ uniquely in the form $u+\epsilon v$ where $u\in\GL(\V)(\R)$ and $v\in {\rm End}_{\R-{\rm mod}}(\V(\R))$, where for $x+\epsilon y\in\V(\R[\epsilon])=\V(\R)\oplus\epsilon\V(\R)$, we have 

\begin{equation}
(u+\epsilon v)(x+\epsilon y)=u(x)+\epsilon (u(y)+v(x)).
\label{eq1}\end{equation}
(ii) The Lie algebra $\gl(\V)(\R)$ is then identified with $\{1+\epsilon v\,|\, v\in{\rm End}_\R(\V(\R))\}\simeq {\rm End}_\R(\V(\R))$ and the adjoint action is the conjugation action of $\GL(\V)(\R)$ on ${\rm End}_\R(\V(\R))$.

\noindent (iii) An element $u+\epsilon v\in \GL(\V)(\R[\epsilon])$ corresponds to the element $(vu^{-1},u)\in \gl(\V)(\R)\rtimes\GL(\V)(\R)$. For $(w,g)\in \gl(\V)(\R)\rtimes\GL(\V)(\R)$ and $(x+\epsilon y)\in\V(\R[\epsilon])$, the equation $$(w,g)(x+\epsilon y)=g(x)+\epsilon \left(w(g(x))+g(y)\right),$$corresponds to (\ref{eq1}) with $(w,g)=(vu^{-1},u)$.

\label{GL}\end{remark}

\bigskip

Consider the $\k$-group functor $\V^\vee$  associated with the $\k$-vector space $\V^\vee={\rm Hom}_\k(\V,\k)$. Namely, for any $\k$-algebra $\R$ 

 $$
\V^\vee(\R):= \V^\vee\otimes_\k\R\simeq\Hom_\k(\V,\R)\simeq \Hom_\R(\V(\R),\R)=\V(\R)^\vee.
 $$
 The first isomorphism is \cite[\S 4, Proposition 2]{Bour} and the second one follows from the fact that the restriction map $\Hom_\R(\V(\R),\R)\rightarrow\Hom_\k(\V,\R)$ is the inverse map of $\Hom_\k(\V,\R)\rightarrow \Hom_\R(\V(\R),\R)$, $f\mapsto (v\otimes r\mapsto rf(v))$.

Consider a homomorphism of $\k$-group schemes $\rho:\G\to\GL(\V)$ with corresponding Lie algebra homomorphism $\varrho:={\rm Lie}(
\rho):\g\to \gl(\V)$ (c.f. \cite[\S 4.2]{DG}).

%\todo{Tamas: the "corresponding Lie algebra homomorphism" needs to be defined}
%\todo{Emmanuel : it is obvious and standard, $\rho$ provides a group homomorphisms $\rho_{\R[\epsilon]}$ and $\rho_\R$ which makes a commutative diagram, it is immediate that $\g(\R)$  defined as the kernel of $\G(\R[\epsilon])\rightarrow\G(\R)$ is mapped via $\rho$ to the kernel of $\GL(\V)(\R[\epsilon]) \rightarrow \GL(\V)(\R)$}

 We have the dual representation $\rho^\vee$ % \todo{Tamas: this has to be defined} 
of $\G$ on $\V^\vee$ defined as
 
 $$
 \rho^\vee_\R(g)(f)(v)=f(\rho_\R(g^{-1})(v)),
 $$
 for all $g\in\G(\R)$, $v\in \V(\R)$ and $f\in\V^\vee(\R)$. It is equivalent to ask that $$\left\langle \rho_\R(g)(v), \rho_\R^\vee(g)(f)\right\rangle=\langle v,f\rangle,$$ for all $g,v$ and $f$ as above, where $\langle,\rangle$ denotes the natural pairing $\V(\R)\times \V^\vee(\R)\to \R$.  
 
  For a $\k$-algebra $\R$ we define the usual moment map 

$$
\mu_\R:\V(\R)\times \V(\R)^\vee\to \g(\R)^\vee
$$

by 

$$
\mu_\R(v,f)=(\delta_v)^\vee(f)
$$
where for $v\in\V(\R)$ we put $$ \begin{array}{cccc}\delta_v:&\g(\R)&\rightarrow&\V(\R)\\ &x&\mapsto&\varrho_\R(x)(v)\end{array}$$ which is  $\R$-linear, as $\varrho_\R={\rm Lie}(\rho)_\R$ is $\R$-linear. 
 Thus for any $(v,f)\in \V(\R)\times \V(\R)^\vee$ and $x\in \g(\R)$  we have \beq{momentdef} \mu_\R(v,f)(x)= \langle \varrho_\R(x)(v),f \rangle. \eeq We note that $\mu_\R$ is equivariant
  with respect to the natural action of $\G(\R)$, namely for $g\in\G(\R)$
 
 \begin{align}\nonumber
 \mu_\R\left(\rho_\R(g)(v),\rho^\vee(g)(f)\right))&=\left\langle\varrho_\R(x)(\rho_\R(g)(v)),\rho^\vee_\R(g)(f)\right\rangle\\& \nonumber =\left\langle\rho_\R(g^{-1})\circ\varrho_\R(x)\circ\rho_\R(g)(v),f\right\rangle\\&=\left\langle\varrho_\R({\rm Ad}(g^{-1})(x))(v),f\right\rangle \nonumber\\ &= \mu_\R(v,f)({\rm Ad}(g^{-1})(x))
 \label{equimoment}  \end{align}
where ${\rm Ad}$ denotes the adjoint action of $\G$ on $\g$.  Hence it makes sense to consider the orbit space $\mu_\R^{-1}(0)/\G(\R)$.

%\begin{remark}If $\R$ is a $\k$-Frobenius algebra with $1$-form $\lambda:\R\rightarrow\k$. Then we have an isomorphism of $\R$-modules

%$$
%\D(\V)(\R)\rightarrow \V(\R)^*={\rm Hom}_\k(\V(\R),\k),
%$$ \todo{Tamas: What is $\D(\V)$? has to be defined}
%\todo{Tamas: also what is the role of this Remark 4.2?}\todo{Emmanuel: This remark was interesting with Frobenius ring but anymore, I suggest to erase it}
%mapping $f\in{\rm Hom}_\k(\V,\R)$ to $\tilde{f}:\V(\R)\rightarrow\k$ defined by $\tilde{f}(v\otimes r)=\lambda(rf(v))$ for $r\in \R$ and $v\in\V$. The moment map can be re-written as \bes \begin{array}{cccc} \mu_\R:&
%\V(\R)\times\V(\R)^*&\rightarrow&\g(\R)^*\\ &(v,f)&\mapsto&(\delta_v)^*(f)\end{array}\ees
%\end{remark}

\subsection{Fourier transforms}

In this section we assume that $\R$ is a $\k$-Frobenius algebra with $1$-form $\lambda:\R\rightarrow\k$, $\V$ is a finite dimensional $\k$-vector space and $\rho:\G\rightarrow\GL(\V)$ a representation with infinitesimal action $\varrho:\g\rightarrow\gl(\V)$. 
%\todo{Tamas: Why is $\varrho_\R$ used here, why not only $\varrho$?}\todo{Emmanuel: This was a misprint, I corrected it} 
In the following we make the identifications $\G(\R[\epsilon])\simeq\g(\R)\rtimes\G(\R)$ and $\V(\R[\epsilon])=\V(\R)\times\V(\R)$ where the second coordinate is the $\epsilon$-coordinate. We also use the simplified notations $$g\cdot y:=\rho_\R(g)(y),$$ and $$g\cdot f:= \rho^\vee_\R(g)(f)$$ with $g\in\G(\R)$, $y\in\V(\R)$ and $f\in \V^\vee(\R)$. % \todo{Tamas; what is $y$ and $g$ here? should not we use $\rho_\R$ instead of $\rho$?} \todo{Emmanuel: I corrected it}
By Remark~\ref{GL} (iii), the action of $\G(\R[\epsilon])$ on $\V(\R[\epsilon])$ reads

\begin{equation}
(x,g)\cdot(x_1,x_2)=(g\cdot x_1,\varrho_\R(x)(g\cdot x_1)+g\cdot x_2),
\label{actionfor}\end{equation}
for all $(x,g)\in\G(\R[\epsilon])$ and $(x_1,x_2)\in\V(\R[\epsilon])$. % This action of $\G(\R[\epsilon])$ on $\V(\R[\epsilon])$ induces an action on its space of functions $\C[\V(\R[\epsilon])]$.

 We will also need the following formula (the first equation follows from \eqref{momentdef} and the second one from equivariance of the moment map \eqref{equimoment})
 %\todo{Tamas: this needs to be spelled out. The first equation follows from \eqref{momentdef} and the second should follow from the equivariance of the moment map for which there is a todo above})

\begin{equation}
\left\langle \varrho_\R(x)(g\cdot x_1),g\cdot x_3\right\rangle=\mu_\R(g\cdot x_1,g\cdot x_3)(x)=\mu_\R(x_1,x_3)({\rm Ad}(g^{-1})(x)),
\label{momentfor}\end{equation}
for all $x\in\g(\R)$, $g\in\G(\R)$, $x_1\in\V(\R)$ and $x_3\in\V(\R)^\vee$, where $\langle\,,\,\rangle:\V(\R)\times\V(\R)^\vee\rightarrow\R$  is the natural pairing.

We now assume that $\k$ is a finite field and we fix a non-trivial additive character $\psi:\k\rightarrow\C^\times$ and we put $\varpi=\psi\circ\lambda:\R\rightarrow\C^\times$.  Define an action of $\G(\R[\epsilon])$ on the space of functions $\C[\V(\R)\times\V(\R)^\vee]$ by

\beq{action}
((x,g)^{-1}\cdot f)(x_1,x_3):=\varpi\left(-\mu_\R(x_1,x_3)({\rm Ad}(g^{-1})(x))\right)f(g\cdot x_1,g\cdot x_3),
\eeq
for all $(x,g)\in\G(\R[\epsilon])$, $(x_1,x_3)\in\V(\R)\times\V(\R)^\vee$ and $f\in\C[\V(\R)\times\V(\R)^\vee]$.

We now consider the Fourier transform $\mathcal{F}:\C[\V(\R)\times\V(\R)]\rightarrow\C[\V(\R)\times\V(\R)^\vee]$ as defined in \S \ref{Fouriergeneral}  with $X=\V(\R)$ and $\M=\V(\R)$. Namely

$$
\mathcal{F}(f)(x_1,x_3)=\sum_{x_2\in\V(\R)}\varpi\left(\langle x_2,x_3\rangle\right)f(x_1,x_2),
$$
for all $f\in\C[\V(\R)\times\V(\R)]$ and $(x_1,x_3)\in\V(\R)\times\V(\R)^\vee$.

\begin{proposition} $\mathcal{F}$ is a $\G(\R[\epsilon])$-equivariant isomorphism. \label{mainprop}
\end{proposition}

\begin{proof}We have 

\begin{align*}\mathcal{F}\left((x,g)^{-1}\cdot f\right)(x_1,x_3)&=\sum_{x_2}\varpi(\langle x_2,x_3\rangle)((x,g)^{-1}\cdot f)(x_1,x_2)\\
&=\sum_{x_2}\varpi(\langle x_2,x_3\rangle)f(g\cdot x_1,\varrho_\R(x)(g\cdot x_1)+g\cdot x_2)
\end{align*}
After the change of variables

$$
x'_2=\varrho_\R(x)(g\cdot x_1)+g\cdot x_2,
$$
we find

\begin{align*}\mathcal{F}\left((x,g)^{-1}\cdot f\right)(x_1,x_3)&=\sum_{x_2'}\varpi\left(\langle x'_2-\varrho_\R(x)(g\cdot x_1),g\cdot x_3\rangle\right) f(g\cdot x_1,x'_2)\\
&=\varpi\left(\langle-\varrho_\R(x)(g\cdot x_1),g\cdot x_3\rangle\right)\sum_{x'_2}\varpi(\langle x'_2,g\cdot x_3\rangle)f(g\cdot x_1,x_2')\\
&=\varpi(-\mu_\R(x_1,x_3)({\rm Ad}(g^{-1}(x))))\,\mathcal{F}(f)(g\cdot x_1,g\cdot x_3)\\
&=\left((x,g)^{-1}\cdot\mathcal{F}(f)\right)(x_1,x_3).
\end{align*}
The third identity follows from (\ref{momentfor}).
\end{proof}

We regard $\g(\R)$ as a subgroup of $\G(\R[\epsilon])$ from which we get an action of $\g(\R)$ on $\C[\V(\R)\times\V(\R)^\vee]$ by restricting the action \eqref{action}. The next result relates the  isotypical components of $\C[\V(\R)\times\V(\R)^\vee]$ with respect to linear characters of $\g(\R)$ and the fibers of the moment map $\mu_\R$.

First of all notice that any linear (additive) character $\g(\R)\rightarrow\C^\times$ can be written uniquely in the form $\varpi\circ \xi$ for some  $\xi\in\g(\R)^\vee$. If $H$ is a finite commutative group acting on a finite dimensional $\C$-vector space $V$, and if $\chi$ is a linear character of $H$, we denote by $V_\chi$ the  $\chi$-isotypical component, namely

$$
V_\chi=\{f\in V\,|\, h\cdot f=\chi(h)\, f\text{ for all }h\in H\}.
$$

\begin{lemma}Let $\xi\in\g(\R)^\vee$. Then $$\C[\V(\R)\times\V(\R)^\vee]_{\varpi\circ\xi}=\C[\mu_\R^{-1}(\xi)],$$ where we identify $\C[\mu_\R^{-1}(\xi)]$ with the $\C$-vector subspace of $\C[\V(\R)\times\V(\R)^\vee]$ of functions supported on $\mu_\R^{-1}(\xi)$.
\label{lemmaiso}\end{lemma}

\begin{proof} Suppose $f\in\C[\V(\R)\times\V(\R)^\vee]_{\varpi\circ\xi}$, and let $(x_1,x_3)\in\V(\R)\times\V(\R^\vee)$. For any $x\in\g(\R)$ we have

$$
(x\cdot f)(x_1,x_3)=(\varpi\circ\xi)(x)\, f(x_1,x_3).
$$
But also by (\ref{action}) we have 

$$
(x\cdot f)(x_1,x_3)=\varpi(\mu_\R(x_1,x_3)(x))\, f(x_1,x_3).
$$
If $f(x_1,x_3)\neq 0$, then for all $x\in\g(\R)$ we have $$\varpi(\xi(x))=\varpi(\mu_\R(x_1,x_3)(x)),$$ i.e. $$\xi(x)-\mu_\R(x_1,x_3)(x)\in{\rm Ker}(\varpi).$$  As $\xi(x)-\mu_\R(x_1,x_3)(x)$ is $\R$-linear in $x$ the image $\{\xi(x)-\mu_\R(x_1,x_3)(x) : x\in \g(\R)\}$ is an ideal of $\R$. Since the kernel of $\varpi$ does not contain non-trivial ideals of $\R$ we must have $$\xi(x)-\mu_\R(x_1,x_3)(x)=0$$ for all $x\in\g(\R),$ that is $(x_1,x_3)\in\mu_\R^{-1}(\xi)$. We proved that $\C[\V(\R)\times\V(\R)^\vee]_{\varpi\circ\xi}\subset\C[\mu^{-1}_\R(\xi)]$ and the other inclusion is straigthforward.

\end{proof}

From Proposition \ref{mainprop} and Lemma \ref{lemmaiso} we deduce

\begin{corollary}Let $\xi\in\g(\R)^\vee$. Then $\mathcal{F}$ restricts to an isomorphism of ${\rm Stab}_{\G(\R)}(\xi)$-modules 

\begin{equation}\C[\V(\R)\times\V(\R)]_{\varpi\circ\xi}\simeq \C[\mu_\R^{-1}(\xi)].\label{mainiso}\end{equation}

If $\xi$ is central, then 

\begin{equation}
\card\left(\mu_\R^{-1}(\xi)/\G(\R)\right)={\rm dim}\, \C[\V(\R)\times\V(\R)]_{\varpi\circ \xi}^{\G(\R)},
\label{mainfor1}\end{equation}
where the right hand side is the dimension of the $\G(\R)$-invariant subspace of $\C[\V(\R)\times\V(\R)]_{\varpi\circ \xi}$. 

In particular for $\xi=0$, we have 

\begin{equation}
\card \left(\mu_\R^{-1}(0)/\G(\R)\right)=\card \left(\V(\R[\epsilon])/\G(\R[\epsilon])\right),
\label{mainfor}\end{equation}

\label{maincoro}\end{corollary}

Formula (\ref{mainfor1}) is obtained from the isomorphism (\ref{mainiso}) by taking the $\G(\R)$-invariant parts.

\begin{remark} We can compute the dimension of $\C[\V(\R)\times\V(\R)]_{\varpi\circ\xi}$ in \eqref{mainiso} by taking inner product of the character $\chi_{\bar{\varrho}_{R}}$ of the representation  $\bar{\varrho}_{R}$ of the finite additive group $\g(R)$ on  $\C[\V(\R)\times\V(\R)]$ with the character ${\varpi\circ\xi}$. Recalling that the action  $\bar{\varrho}_{\R}$ is induced by restricting \eqref{actionfor} to $\g(\R)\subset \G(\R[\epsilon])$, i.e. setting $g=1$, we obtain
\bes \dim(\C[\V(\R)\times\V(\R)]_{\varpi\circ\xi})&=&\langle\chi_{\bar{\varrho}_{R}},{\varpi\circ\xi}\rangle\\& =& \frac{1}{|\g(\R)|}\sum_{x\in \g(\R)  }\Tr\,({\bar{\varrho}_{R}}(-x))( \varpi\circ\xi)(x)\\&=& \frac{1}{|\g(\R)|}\sum_{x\in \g(\R)  }|\V(R)||\ker(\varrho_\R(x))|( \varpi\circ\xi)(x).\ees Thus \eqref{mainiso} implies that 
\begin{equation}\label{dimension}\dim(\C[\mu_\R^{-1}(\xi)])=|\mu_\R^{-1}(\xi)|=\frac{1}{|\g(\R)|}\sum_{x\in \g(\R)  }|\V(\R)||\ker(\varrho_\R(x))|( \varpi\circ\xi)(x).\end{equation}
When $\R=\k$ a finite field \eqref{dimension} is \cite[Proposition 2]{hausel-kac}, which was the count formula used to determine the Poincar\'e polynomials of Nakajima quiver varieties. Thus \eqref{mainiso} could be considered a categorification of \cite[Proposition 2]{hausel-kac}.
\end{remark}

\subsection{A counter-example}\label{counterex}

We give  counter-examples to Formula (\ref{mainfor}) when $\R$ is not a Frobenius $\k$-algebra.

Consider $\R=\F_q[t_1,\dots,t_n]/(t_1,\dots,t_n)^2$. It is not a Frobenius $\F_q$-algebra for $n>1$. We put $$\G:=\mathbb{G}_m={\rm Spec}(\F_q[t,t^{-1}]),$$ and we consider the action of $\G$ on $\mathbb{V}=\mathbb{A}^1={\rm Spec}(\F_q[t])$ by multiplication. 

Then   $\V(\R[\epsilon])/\G(\R[\epsilon])$ is the set of $\R[\epsilon]^\times$-orbits on $\R[\epsilon]$ for the multiplication action. Then

$$
A_n(q):=\card\,\left(\R[\epsilon]/\R[\epsilon]^\times\right)=\frac{1}{\card\left(\R[\epsilon]^\times\right)}\card\left\{(a,b)\in\R[\epsilon]^\times\times\R[\epsilon],\, (a-1)b=0\right\}.
$$
Write $a=a_0+A+\epsilon(a_0'+A')$ with $a_0,a_0'\in\F_q$ and $A,A'\in(t_1,\dots,t_n)$, similarly for $b$.

The equation $(a-1)b=0$ is equivalent to the following system in $\k$

$$
\begin{cases}(a_0+A)(b_0+B)=(b_0+B),\\
(a_0+A)(b_0'+B')+(a_0'+A')(b_0+B)=b_0'+B'.
\end{cases}
$$
We have the following cases :

(i) If $a_0\neq 1$, then $a-1$ is invertible and so we must have $b=0$. The number of solutions of $ab=b$ is thus $(q-2)q^{2n+1}$.

(ii) If $a_0=1$, $a_0'=0$, $A\neq 0$, then $b_0=b_0'=0$ and $B,B'$ can be anything, and the number of solutions of $ab=b$ is $q^{3n}(q^n-1)$.

(iii) If $a_0=1$, $a_0'=0$, $A=0$ and $A'\neq 0$, then $b_0=0$ and $b_0', B, B'$ can be anything. In this case  the number of solutions of $ab=b$ is $q^{2n+1}(q^n-1)$.

(iv) If $a=1$, then $b$ can be anything, and so  the number of solutions of $ab=b$ is $q^{2n+2}$.

(v)  If $a_0=1$, $a_0'\neq 0$, then $b_0=0$ and $B$ is determined, and  the number of solutions of $ab=b$ is $(q-1)q^{3n+1}$.
\bigskip

The total number of orbits is then

\begin{align*}
A_n(q)&=\frac{1}{(q-1)q^{2n+1}}\left((q-2)q^{2n+1}+q^{3n}(q^n-1)+q^{2n+1}(q^n-1)+q^{2n+2}+(q-1) q^{3n+1}\right)\\
&=2+q^n+\sum_{i=0}^{n-1}q^{i+n-1}+\sum_{i=0}^{n-1}q^i.
\end{align*}
\bigskip

We now compute the number of $\G(\R)$-orbits of $\mu_\R^{-1}(0)$. The moment map $\mu_\R:\V(\R)\times\D(\V)(\R)\rightarrow\D(\g(\k))(\R)$  is just the multiplication $\R\times\R\rightarrow\R$  and the action of $\G(\R)=\R^\times$ is given by $\lambda\cdot(a,b)=(\lambda a,\lambda^{-1} b)$. We compute the number $B_n(q)$ of $\R^\times$-orbits of $\{(a,b)\in\R\times\R\,|\, ab=0\}$.

Write an element $a\in\R$ in the form $a_0+A$ with $a_0\in\F_q$ and $A\in(t_1,\dots,t_n)$. The equation $ab=0$ is equivalent to the system

$$
\begin{cases}a_0b_0=0,\\
a_0B+b_0A=0.\end{cases}
$$
We have the following cases :
\bigskip

(i) If $a_0\neq 0$ (resp. $b_0\neq 0$), then $b=0$ (resp. $a=0$) , and so there is only one $\R^\times$-orbit in this case.

(ii) If $a_0=b_0$ and $A,B$ both non-zero, then there are exactly $(q^n-1)^2/(q-1)$ orbits.

(iii) If  $a=0$, $b_0=0$ and $B\neq 0$ (resp. $b=0$, $a_0=0$ and $A\neq 0$) then there are $(q^n-1)/(q-1)$ orbits.

(iv) If $a=b=0$, there is one orbit.
 \bigskip
 
 We then find
 
 $$
 B_n(q)=3+(q-1)\left(\sum_{i=0}^{n-1}q^i\right)^2+2\left(\sum_{i=0}^{n-1}q^i\right).
 $$
 A simple calculation shows that $B_n-A_n =
 (q^n-1)(q^{n-1}-1)$. Therefore the equality $A_n=B_n$ holds only for
 $n=1$.  In fact only in this case $\R$ is a finite Frobenius
 algebra and the equality $A_1=B_1$ is hence an instance of our main theorem.

\section{Representations of quivers}

We use the notation of the introduction.

\subsection{The finite field case}

Here we prove Theorems \ref{main1} and\ref{maintheoquiv} in the finite field case. We thus assume that $\k$ is a finite field  and that $\R$ is a Frobenius $\k$-algebra.

\begin{proposition} 

\noindent (i) Assume that $Q'$ is a quiver obtained from $Q$ by changing the orientation of some arrows. Then there exists an isomorphism $\C[\rep^\alpha(Q,\R)]\rightarrow\C[\rep^\alpha(Q',\R)]$ of $\G_\alpha(\R)$-modules.

\noindent (ii) Let $\mu_\R=\mu_{R,\alpha}$ be the moment map (\ref{moment}). There exists an isomorphism $\C[\rep^\alpha(Q,\R[\epsilon])]^{\g_\alpha(\R)}\rightarrow\C[\mu_{\R,\alpha}^{-1}(0)]$ of $\G_\alpha(\R)$-modules.
\label{proprep}\end{proposition}

\begin{proof} The assertion (ii) is a particular case of the first assertion of Corollary \ref{maincoro} with $\V:=\rep^\alpha(Q,\k)$, $\G=\G_\alpha$ and $\xi=0$. The implication of Assertion (i) on the dimension of invariants in the case where $\R=\k$ is a result of Kac. The proof given in \cite[\S 5.5]{kraft-riedtmann} in the finite field case still work in the Frobenius $\k$-algebra case. It can actually be simplified as we now see. The proof reduces to the case where $Q'$ is obtained from $Q$ by inverting the orientation of one arrow $a:i\rightarrow j$ only. We then put 

$$
\H={\rm Hom}(\R^{\alpha(i)},\R^{\alpha(j)}), \hspace{1cm} \H'={\rm Hom}(\R^{\alpha(j)},\R^{\alpha(i)}),\hspace{1cm}\K=\bigoplus_{k\rightarrow l\in E\backslash\{a\}} {\rm Hom}_\R(\R^{\alpha(k)},\R^{\alpha(l)}).
$$
We have $\rep^\alpha(Q,\R)=\K\oplus \H$ and $\rep^\alpha(Q',\R)=\K\oplus \H'$. The assertion (i) is then a consequence of the results of \S \ref{Fouriergeneral} with $X=\K$, $\M=\H$ and where we identified $\H'$ with $\H^\vee$ via the perfect pairing

$$
\H'\times \H\rightarrow \R, \hspace{1cm} (f,g)\mapsto {\rm Tr}(g\circ f).
$$
The above pairing is perfect as the multiplication $\R\times\R\rightarrow \R$ is perfect. Indeed, if $ab=0$ for all $a\in \R$, then $\lambda(ab)=0$ for all $a\in\R$ and so $\lambda$ vanishes on the ideal generated by $b$ which is impossible as the kernel of $\lambda$ does not contain non-trivial ideals of $\R$.
\end{proof}

\begin{corollary} 

\noindent (i) Assume that $Q'$ is a  quiver obtained from $Q$ by changing the orientation of some arrows. Then 

$$
m_\alpha(\R Q)=m_\alpha(\R Q').
$$

\noindent (ii) We have 

$$
m_\alpha(\Pi_\R(Q))=m_\alpha(\R[\epsilon]Q).
$$
\label{finitecase}\end{corollary}

\begin{proof}The assertion (ii) is a reformulation of Formula (\ref{mainfor}). The assertion (i) is obtained by taking the $\G_\alpha(\R)$-invariant part of the $\G_\alpha(\R)$-equivariant isomorphism in Proposition \ref{proprep} (i).  
\end{proof}

We now want to prove the analogue of the above corollary for absolutely indecomposable representations. Instead of taking the $\G_\alpha(\R)$-invariant part (i.e. the isotypical component of the trivial character) we will take the isotypical component of another linear character of $\G_\alpha(\R)$. This approach is very different from Kac's proof \cite{kac} of independence of the orientation in the finite field case (i.e. when $\R=\k$) but is not new in the finite field case (see  \cite[Theorem 1.1]{Let}).

We assume  now in the remainder of this section that $\R$ is a split local Frobenius $\k$-algebra with maximal ideal $\mathfrak{m}$.

If $M\in\rep^\alpha(Q,\R)$, then 

$$
{\rm End}(M)=\left.\left\{(X_i)_{i\in I}\in\bigoplus_{i\in I}\gl(\R^{\alpha(i)})\,\right|\, M_{i\rightarrow j}X_i=X_j M_{i\rightarrow j}\text{ for all }i\rightarrow j\in E\right\}.
$$
Note that 

$$
{\rm End}(M)^\times={\rm Stab}_{\G_\alpha(\R)}(M).
$$

In the following we put $\overline{\R}:=\R\otimes_\k\overline{\k}$. As $\overline{\R}$ is finite-dimensional over $\overline{\k}$, any $M'\in\rep^\alpha(Q,\overline{\R})$ is a finite-dimensional $\overline{\R}Q$-module and so we have the following well-known result \cite[Theorem 1.15]{Brion}.

\bigskip

\begin{theorem}Let

$$
M'\simeq\bigoplus_{i=1}^rm_i M'_i
$$
be the decomposition of $M'\in\rep^\alpha(Q,\overline{\R})$ into indecomposables, then 

$$
{\rm End}(M')= I_{M'}\oplus K,
$$
with $I_{M'}$ a nilpotent ideal and  $K$ a $\overline{\k}$-subalgebra of ${\rm End}(M')$ isomorphic to $\prod_{i=1}^r{\rm Mat}_{m_i}(\overline{\k})$.
\label{proprappel}\end{theorem}

If $M\in\rep^\alpha(Q,\R)$, we denote by $\overline{M}$ the representation of $Q$ over $\overline{\R}$ obtained from $M$ by extension of scalars from $\k$ to $\overline{\k}$. Theorem \ref{proprappel} has the following consequence.

\begin{corollary}$M\in\rep^\alpha(Q,\R)$ is absolutely indecomposable if and only if ${\rm End}(\overline{M})= I_{\overline{M}}\oplus\mathfrak{z}$ where $\mathfrak{z}=\mathfrak{z}(\overline{\k}):=\{(\lambda\, {\rm Id}_{\alpha(i)})_{i\in I}\,|\, \lambda\in\overline{\k}\}\simeq\overline{\k}$ and ${\rm Id}_{\alpha(i)}$ the identity endomorphism of $\overline{\R}{^{\alpha(i)}}$.
\label{coro}
\end{corollary}

To speak of Jordan decomposition \cite[Theorem 9.18]{milne}  we regard $\G_\alpha(\overline{\R})$ as an affine $\overline{\k}$-algebraic group and $\G_\alpha(\R)$ as the group of $\k$-points of $\G_\alpha(\overline{\R})$. 

More precisely, we consider the affine $\k$-group scheme $\G_\alpha^{\R/\k}$ which represents the $\k$-group functor 
$$
\Alg_\k\rightarrow\Gr,\hspace{1cm} \R'\mapsto \G_\alpha(\R\otimes_\k\R')
$$
so that $\G_\alpha^{\R/\k}(\k)=\G_\alpha(\R)$ and $\G_\alpha^{\R/\k}(\overline{\k})=\G_\alpha(\overline{\R})$. This functor is indeed representable as $\R$ is  finite-dimensional over $\k$ (see \cite[Theorem 4]{Neron}).

Let $\pi:\G_\alpha^{\R/\k}\rightarrow\G_\alpha$ be induced by the $\k$-algebra morphism $\R\rightarrow\k$ and denote by $\U_\alpha$ the kernel of $\pi$. Then $\G_\alpha^{\R/\k}=\U_\alpha\rtimes\G_\alpha$ and $\U_\alpha$ is a unipotent group. 

For any $\k$-algebra $\R'$, we denote by $Z_\alpha(\R')\simeq\mathbb{G}_m(\R')$ the subgroup of $\G_\alpha(\R')$  of elements of the form $(\lambda\, {\rm Id}_{\alpha(i)})_{i\in I}$ for some $\lambda\in\R'{^\times}$ where ${\rm Id}_{\alpha(i)}$ denotes the identity of $\GL(\R'^{\alpha(i)})$. The set of semisimple elements of $Z_\alpha(\R)$ and $Z_\alpha(\overline{\R})$ are respectively $Z_\alpha(\k)$ and $Z_\alpha(\overline{\k})$.

\begin{proposition}Let $M\in\rep^\alpha(Q,\R)$. The following assertions are equivalent :

\noindent (i) $M$ is absolutely indecomposable.

\noindent (ii) The semisimple elements of ${\rm Stab}_{\G_\alpha(\overline{\R})}(\overline{M})$ all lie in $Z_\alpha(\overline{\k})$.

\noindent (iii) The semisimple elements of ${\rm Stab}_{\G_\alpha(\R)}(M)$ all lie in $Z_\alpha(\k)$.
\label{equiv}\end{proposition}

\begin{proof}The equivalence between (i) and (ii) is a reformulation of Corollary \ref{coro} using the fact that ${\rm Stab}_{\G_\alpha(\overline{\R})}(\overline{M})={\rm End}(\overline{M})^\times$. The implication (ii) $\Rightarrow$ (iii) is clear as $Z_\alpha(\overline{\k})\cap \G_\alpha(\R)=Z_\alpha(\k)$. For the implication (iii) $\Rightarrow$ (ii), notice that if (ii) is not true then  the maximal tori of ${\rm Stab}_{\G_\alpha(\overline{\R})}(\overline{M})$ are of rank at least $2$, and so ${\rm Stab}_{\G_\alpha(\R)}(M)$ contains the $\k$-points of a two-dimensional torus which contradicts (iii).

\end{proof}

We construct a linear character $1^\chi$ of $\G_\alpha(\R)$ as follows. We need to assume that $\k$ contains a primitive $|\alpha|$-th root of unity with $|\alpha|=\sum_{i\in I}\alpha(i)$. Let $\chi:\k{^\times}\rightarrow\C^\times$ be a linear character of order $|\alpha|$. We consider the linear character $1^\chi:=\chi\circ{\rm det}\circ\pi(\k):\G_\alpha(\R)\rightarrow \C^\times$ where ${\rm det}:\G_\alpha(\k)\rightarrow\k{^\times}$ is given by  $(g_i)_{i\in I}\mapsto \prod_{i\in I}{\rm det}(g_i)$.
\bigskip

\begin{lemma} Let $X$ be either  $\rep^\alpha(Q,\R)$, $\rep^\alpha(Q,\R[\epsilon])$ or $\mu^{-1}_{\R,\alpha}(0)\subset\rep^\alpha(\overline{Q},\R)$. Consider the map

$$
\{(M,g)\in X\times\G_\alpha(\R)\,|\, g\cdot M=M\}\rightarrow X,\hspace{1cm}(M,g)\mapsto M.
$$
Then $1^\chi$ defines a character (by restriction) on the fibers of this map, and this character is trivial exactly on the fibers of the absolutely indecomposable representations.

\label{keylemma}\end{lemma}

\begin{proof}Notice that the restriction of $1^\chi$ to the subgroup $Z_\alpha(\k)$ is the trivial character. Therefore, by Proposition \ref{equiv}, the restriction of $1^\chi$ to the stabilizer of an absolutely indecomposable representation of $Q$ over $\R$ is trivial. Conversely if $M\in\rep^\alpha(Q,\R)$ decomposes non trivially as $M=M_1\oplus M_2$ with $M_1\in\rep^{\alpha_1}(Q,\R)$ and $M_2\in\rep^{\alpha_2}(Q,\R)$, then ${\rm Stab}_{\G_\alpha(\R)}(M)$ contains $Z_{\alpha_1}(\k)\times Z_{\alpha_2}(\k)$, and so the restriction of $1^\chi$ to ${\rm Stab}_{\G_\alpha(\R)}(M)$ is non-trivial as it is non-trivial on $Z_{\alpha_1}(\k)\times Z_{\alpha_2}(\k)$. Let us generalize this argument to non absolutely indecomposable representations. Assume that $M\in\rep^\alpha(Q,\R)$ is not absolutely indecomposable and let $\overline{M}=\bigoplus_{i=1}^rm_iM_i'$ be its decomposition into indecomposables with $M'_i\in\rep^{\alpha_i}(Q,\overline{\R})$. Denote by $F$ the Frobenius endomorphism on $\rep^\alpha(Q,\overline{\R})$ and $\G_\alpha(\overline{\R})$ associated with the $\k$-structure. As $\overline{M}$ is $F$-stable, the Frobenius permutes the blocks $m_iM_i'$. We now re-decompose $\overline{M}$ as $N_1'\oplus\cdots\oplus N_s'$ with $F(N_i')=N_i'$ for all $i=1,\dots,s$, where $F$ permutes cyclically the indecomposable blocks inside each $N_i'$. Then each $N_i'$ is of the form $\overline{N}_i$ for some $N_i\in\rep^{\alpha_i}(Q,\R)$. Hence if $s\geq 2$, then $M$ is decomposable over $\R$ and as above we prove that the restriction of $1^\chi$ on the stabilizer of $M$ is not trivial. We thus assume that $s=1$ and (without loss of generality) that $m_1=\cdots =m_r=1$. We must have $\alpha_1=\dots=\alpha_r=:\beta$ and re-indexing if necessary, we can assume that $M'_{i+1}=F(M_i')$ for $i=1,\dots,r-1$. We thus have $F^r(M'_i)=M'_i$ for all $i$ and  $F({\rm Stab}_{\G_\beta(\overline{\R})}(M_i'))={\rm Stab}_{\G_\beta(\overline{\R})}(M_{i+1}')$. Let $\k'$ be the subfield of $\overline{\k}$ of elements fixed by $F^r$.  Therefore for any $t\in\mathbb{G}_m(\k')$, the element 

$$(t,F(t),\dots,F^{r-1}(t))\in Z_\beta(\overline{\k})\times\cdots\times Z_\beta(\overline{\k})\subset{\rm Stab}_{\G_\beta(\overline{\R})}(M'_1)\times\cdots\times{\rm Stab}_{\G_\beta(\overline{\R})}(M'_r)$$
 is $F$-stable and so lives in the stabilizer ${\rm Stab}_{\G_\alpha(\R)}(M)$.  Therefore ${\rm Stab}_{\G_\alpha(\R)}(M)$ contains a copy of $\mathbb{G}_m(\k')$ and the restriction of $1^\chi$ on $\mathbb{G}_m(\k')$ (considered as a subgroup of the stabilizer of $M$)  is the character $t\mapsto \chi(N_{\k'/\k}(t))$ where $N_{\k'/\k}$ is the norm map. This character is not trivial unless the field extension $\k'/\k$ is trivial, i.e. unless $r=1$ in which case $M$ is absolutely indecomposable.
\end{proof}

\begin{proposition} We have

\noindent (i) ${\rm dim}_\C\,\C[\rep^\alpha(Q,\R)]_{1^\chi}=a_\alpha(\R Q)$.

\noindent (ii) ${\rm dim}_\C\,\C[\rep^\alpha(Q,\R[\epsilon])]_{1^\chi}=a_\alpha(\R[\epsilon] Q)$, where we regard $1_\chi$ as a linear character of $\G_\alpha(\R[\epsilon])$ using the projection $\G_\alpha(\R[\epsilon])\rightarrow\G_\alpha(\R)$.

\noindent (iii) ${\rm dim}_\C\,\C[\mu^{-1}_{\R,\alpha}(0)]_{1^\chi}=a_\alpha(\Pi_\R Q)$.

\end{proposition}

The idea to express the number of absolutely indecomposable representations in terms of the dimension of the isotypical component $\C[{\rm Rep}^\alpha(Q,R)]_{1^\chi}$ already appears in \cite[Theorem 1.1]{Let} in the finite field case (i.e. when $\R=\k)$.

\begin{proof}The proofs of (i), (ii) and (iii) are completely similar using Lemma \ref{keylemma}. We therefore only prove (i). We have

\begin{align*}
{\rm dim}_\C\,\C[\rep^\alpha(Q,\R)]_{1^\chi}&=\left( \C[\rep^\alpha(Q,\R)],1^\chi\right)_{\G_\alpha(\R)}\\
&=\frac{1}{|\G_\alpha(\R) |}\sum_{g\in\G_\alpha(\R)}\#\{x\in \rep^\alpha(Q,\R)\,|\, g\cdot x=x\}\, 1^\chi(g)\\
&=\frac{1}{|\G_\alpha(\R)|}\sum_{x\in\rep^\alpha(Q,\R)}\sum_{g\in{\rm Stab}_{\G_\alpha(\R)}(x)}1^\chi(g)\\
&=\frac{1}{|\G_\alpha(\R)|}\sum_{x\in\rep^\alpha(Q,\R)}|{\rm Stab}_{\G_\alpha(\R)}(x)|\, \left( 1,1^\chi\right)_{{\rm Stab}_{\G_\alpha(\R)}(x)}\\
&=\frac{1}{|\G_\alpha(\R)|}\sum_{x\in\rep^\alpha_{\rm a.i.}(Q,\R)}|{\rm Stab}_{\G_\alpha(\R)}(x)|\\
&=a_\alpha(\R Q).\\
\end{align*}

The fifth identity follows from Lemma \ref{keylemma}. The last identity is Burnside formula for the counting of orbits of a finite group acting on a finite set.
\end{proof}

\begin{corollary}\noindent (i) Assume that $Q'$ is the quiver obtained from $Q$ by changing the orientation of some arrows. Then 

$$
a_\alpha(\R Q)=a_\alpha(\R Q').
$$

\noindent (ii) We have 

$$
a_\alpha(\Pi_\R(Q))=a_\alpha(\R[\epsilon]Q).
$$
\label{abso}\end{corollary}

\begin{proof} Take the $1^\chi$ isotypical component on both side of the isomorphisms in Proposition \ref{proprep}.
\end{proof}

\subsection{The algebraically closed fields case}\label{algclose}

We assume that $\k$ is an algebraically closed field. Let $\G$ be a connected affine $\k$-algebraic group acting on an affine $\k$-algebraic variety $X$ and let $X_{\rm ind}$ be a $\G$-stable constructible subset of $X$ (typically $X$ will be an affine algebraic variety constructed out of quiver representations and $X_{\rm ind}$ will be the analogous one with quiver representations replaced by indecomposable quiver representations, see Proposition \ref{indec}). Let $Y$ be either $X_{\rm ind}$ or $X$, and put

$$
\I_Y:=\{(x,g)\in Y\times\G\,|\, g\cdot x=x\}.
$$
This is a constructible set. We denote by $c=c(Y)$ the number of irreducible components of $\I_Y$ of maximal dimension and we put

$$
d=d(Y):={\rm dim}\, \I_Y-{\rm dim}\,\G.
$$

\begin{lemma} The set of orbits $Y/\G$ is finite if and only if $d=0$, in which case $\card(Y/\G)=c$.
\label{d=0}\end{lemma}

\begin{proof}To see this we remark that

$$
\I_Y=\coprod_{\mathcal{O}\in Y/\G}\{(x,g)\in\mathcal{O}\times \G\,|\, g\cdot x=x\},
$$
and that for all $\mathcal{O}\in Y/\G$ we have

$$
{\rm dim}\, \{(x,g)\in\mathcal{O}\times \G\,|\, g\cdot x=x\}={\rm dim}\, \G.
$$

\end{proof}

In the following we define 

$$
m(X):=(c(X),d(X)),\hspace{1cm} a(X):=(c(X_{\rm ind}),d(X_{\rm ind})).
$$

We now come back to quiver representations over a Frobenius algebra $\R$ over $\k$. We consider $\rep^\alpha(Q,\R)$, $\rep^\alpha(Q,\R[\epsilon])$ and $\mu^{-1}_{\R,\alpha}(0)$ as $\k$-algebraic varieties, and we put 
 $m_\alpha(\R Q)=m(\rep^\alpha(Q,\R))$ and $m_\alpha(\Pi_\R(Q))=m(\mu^{-1}_{\R,\alpha}(Q))$. If $d=0$, this is just the pair $(\card(Y/G),0)$ by Lemma \ref{d=0}.
\bigskip

In the next proposition we assume that the Frobenius $\k$-algebra $\R$ is local (note that it is split as $\k$ is algebraically closed). 

\begin{proposition}The subsets of indecomposable representations of $\rep^\alpha(Q,\R)$, $\rep^\alpha(Q,\R[\epsilon])$, and $\mu_{\R,\alpha}^{-1}(0)$ are constructible sets over $\k$.
\label{indec}\end{proposition}

\begin{proof}The proof is similar to the case where $\R$ is a field (see for instance \cite[\S 2.5]{kraft-riedtmann}). To see that the same proof works, we need to regard $\rep^\alpha(Q,\R)$, $\rep^\alpha(Q,\R[\epsilon])$ and $\mu_{\R,\alpha}^{-1}(0)$ as algebraic varieties over $\k$ and use that a representation $V$ is indecomposable if and only if  the  nilpotent ideal of ${\rm End}\,(V)$ is of co-dimension $1$ (see Corollary \ref{coro}).
\end{proof}

We then put $a_\alpha(\R Q):=a(\rep^\alpha(Q,\R))$  and $a_\alpha(\Pi_\R(Q)):=a(\mu^{-1}_{\R,\alpha}(0))$.

\subsubsection{The algebraic closure of a finite field}\label{positive}

We assume that $\k=\overline{\F}_q$  for a prime power $q$ and that $\G, X,Y$ and the action of $\G$ are defined over $\F_q$. Enlarging $\F_q$ if necessary we assume that the irreducible components of $\I_Y$ are also defined over $\F_q$. By Burnside's formula 

\begin{equation}
\card\,\left(Y(\F_q)/\G(\F_q)\right)=\frac{\card\, \I_Y(\F_q)}{\card\, \G(\F_q)}.
\label{Burnside}\end{equation}
We have the following theorem.

\begin{theorem}For all integers $r>0$ we have

$$
\card\,\left( Y(\F_{q^r})/\G(\F_{q^r})\right)- cq^{rd}\in\mathcal{O}(q^{r(d-1/2)}).
$$
\label{theoDel}\end{theorem}

\begin{proof} Using Burnside's formula (\ref{Burnside}), the statement reduces to constructible sets and then to algebraic varieties. Let $X$ be an $\F_q$-variety of dimension $d$ such that its irreducible components of dimension $d$ are also defined over $\F_q$. By  the Grothendieck trace formula applied to  $X$ we have

$$
\card\,\left(X(\F_{q^r})\right)=\sum_i(-1)^i\, \Tr\left((F^*) ^r, H_c^i(X,\overline{\Q}_\ell)\right),
$$
for any integer $r>0$, where $F:X\rightarrow X$ is the geometric Frobenius and $H_c^i(X,\overline{\Q}_\ell)$ is the $i$-th $\ell$-adic cohomology group with compact support. Now  $F^*$ acts on the top cohomology $H_c^{2d}(X,\overline{\Q}_\ell)\simeq (\overline{\Q}_\ell)^c$ by multiplication by $q^d$ (this is explained for instance in \cite[\S 6.5]{Springer}). Moreover all complex conjugates of the eigenvalues of $F^*$ on $H_c^i(X,\overline{\Q}_\ell)$, with $i<2d$, have absolute values $\leq q^{d-\frac{1}{2}}$ by Deligne's work on the Weil conjectures \cite{Deligne-Weil}.

\end{proof}

\begin{corollary}Assume that $G',X',Y'$ is another datum analogously to $G,X,Y$. If for all $r>0$ we have

$$
\card\, \left(Y(\F_{q^r})/\G(\F_{q^r})\right)=\card\, \left(Y'(\F_{q^r})/\G'(\F_{q^r})\right),
$$
then $c(Y)=c(Y')$ and $d(Y)=d(Y')$. In particular, if the orbit set $Y(\overline{\F}_q)/G(\overline{\F}_q)$ is finite  then $Y'(\overline{\F}_q)/G'(\overline{\F}_q)$ is also finite and the two cardinalities agree.
\label{parameters}\end{corollary}

\begin{proof}From Theorem \ref{theoDel}, we can read  $c(Y)$ and $d(Y)$ from the term of highest degree in $\card\, (Y(\F_{q^r})/\G(\F_{q^r}))$.

\end{proof}

We now apply the above result to quiver representations.
\bigskip

\begin{theorem} \label{algclofinm} Assume that $\R$ is  a Frobenius $\overline{\F}_q$-algebra. 

\noindent (i) If $Q'$ is a quiver obtained from $Q$ by changing the orientation of some arrows. Then 

$$
m_\alpha(\R Q)=m_\alpha(\R Q').
$$

\noindent (ii) We have 

$$
m_\alpha(\R[\epsilon]Q)=m_\alpha(\Pi_\R(Q)).
$$

\label{theorep}\end{theorem}

\begin{proof}The $1$-form $\lambda:\R\rightarrow\k$ is defined over some finite subfield $\kappa$ of $\k$. Namely there is a finite-dimensional $\kappa$-algebra $\mathcal{R}$ together with a $\kappa$-linear form $\eta:\mathcal{R}\rightarrow\kappa$ that gives back $\lambda$ after extension of scalars from $\kappa$ to $\k$. If the kernel of $\eta$ contains a non-zero ideal, then so does $\lambda$. Therefore $\mathcal{R}$ is a Frobenius $\kappa$-algebra. The $\k$-algebraic varieties  $\rep^\alpha(Q,\R)$, $\rep^\alpha(Q,\R[\epsilon])$, $\mu^{-1}_\R(0)$ are also defined over $\kappa$ and the set of their $\kappa$-points are respectively $\rep^\alpha(Q,\mathcal{R})$, $\rep^\alpha(Q,\mathcal{R}[\epsilon])$ and $\mu^{-1}_{\mathcal{R}}(0)$. The theorem is thus a consequence of Corollaries \ref{finitecase} and \ref{parameters}.
\end{proof}

We now assume that $\k$ contains a primitive $|\alpha|$-th of unity.

\begin{theorem} \label{algclofina} Assume that $\R$ is a local Frobenius $\overline{\F}_q$-algebra. 

\noindent (i) If $Q'$ is a quiver obtained from $Q$ by changing the orientation of some arrows. Then 

$$
a_\alpha(\R Q)=a_\alpha(\R Q').
$$

\noindent (ii) We have 

$$
a_\alpha(\R[\epsilon]Q)=a_\alpha(\Pi_\R(Q)).
$$

\label{theoai}\end{theorem}

\begin{proof}We consider a finite subfield $\kappa$ of $\k$ that contains an $|\alpha|$-th root of unity and such that there is a Frobenius $\kappa$-algebra $\mathcal{R}$ giving $\R$ after extension of scalars from $\kappa$ to $\k$ (see the proof of Theorem \ref{theorep}). The set of $\kappa$-points of the constructible subsets of  indecomposable representations of $\rep^\alpha(Q,\R)$, $\rep^\alpha(Q,\R[\epsilon])$ and $\mu^{-1}_\R(0)$ are the sets of absolutely indecomposable representations of  $\rep^\alpha(Q,\mathcal{R})$, $\rep^\alpha(Q,\mathcal{R}[\epsilon])$ and $\mu^{-1}_{\mathcal{R}}(0)$. Hence the theorem follows from Corollaries \ref{abso} and  \ref{parameters}.
\end{proof}

\subsubsection{Arbitrary algebraically closed field}

We assume that $\k$ is an arbitrary algebraically closed field and that  $\R$ is a Frobenius $\k$-algebra with Frobenius $1$-form $\lambda:\R\rightarrow\k$. Let us write  $\R$ in the form $\k[x_1,\dots,x_r]/(f_1,\dots,f_s)$. Let $\mathcal{B}$ be a $\k$-basis of $\R$  and we denote by ${\rm det}(\lambda)$ the determinant of the non-degenerate bilinear form $\R\times\R\rightarrow\k$, $(a,b)\mapsto \lambda(ab)$ with respect to $\mathcal{B}$. Let $A\subset\k$ be a finitely generated $\Z$-algebra which contains the coefficients of the polynomials $f_j$'s, the coefficients of the matrix of $\lambda$ (with respect to $\mathcal{B}$) and ${\rm det}(\lambda)^{-1}$ (if ${\rm char}(\k)=p>0$, then it is of the form $\F_q[z_1,\dots,z_m]$ for some finite field extension $\F_q$ of $\F_p$). We then denote by $\mathcal{R}$ the $A$-algebra $A[x_1,\dots,x_n]/(f_1,\dots,f_s)$. The elements of $\mathcal{B}$ are clearly independent over $A$ (as they are over $\k$). Enlarging $A$ if necessary we assume that $\mathcal{B}$ generates $\mathcal{R}$, i.e. is a basis of the $A$-module $\mathcal{R}$. Consider the $A$-linear map $\zeta:\mathcal{R}\rightarrow A$ that gives back $\lambda$ after extension of scalars from $A$ to $\k$. 

For any ring homomorphism $\varphi: A\rightarrow K$, with $K$ a field,  we denote by $\mathcal{R}^\varphi$ the $K$-algebra $\mathcal{R}\otimes_AK$. Then the matrix of the bilinear form on $\mathcal{R}^\varphi$ defined from the $1$-form $\zeta^\varphi: \mathcal{R}^\varphi\rightarrow K$, $x\otimes l\mapsto \varphi(\zeta(x))l$ is invertible and so $\mathcal{R}^\varphi$ is a Frobenius $K$-algebra of same dimension as $\R$. 

\begin{lemma}If $\R$ is local, then $\mathcal{R}^\varphi$ is also local. 
\end{lemma}

\begin{proof}If $\mathfrak{m}$ denotes the maximal ideal of $\R$, we may assume that $1\in\mathcal{B}$ and that $\mathcal{B}^*:=\mathcal{B}\backslash\{1\}$ generates $\mathfrak{m}$ as a $\k$-vector space. Then the $A$-submodule of $\mathcal{R}$ generated by $\mathcal{B}^*$ is closed under multiplication. The $K$-vector subspace of $\mathcal{R}^\varphi$ generated by the  elements of $\mathcal{R}^\varphi$ arising from $\mathcal{B}^*$ is thus  also closed under mutliplication, and so is the maximal ideal of $\mathcal{R}^\varphi$.
\end{proof}

The following theorem, together with Theorems \ref{algclofinm} and \ref{algclofina} implies Theorems \ref{indeptheo} and \ref{maintheoquiv} for an arbitrary algebraically closed field $\k$. Indeed, the residue field of a closed point of $A$ is a finite field and so we always have ring homomorphisms from $A$ to the algebraic closure of some finite field.

\begin{theorem} There exists a finitely generated $\Z$-algebra $A\subset\k$ as above and an open dense subset $U$ of ${\rm Spec}(A)$ such that for any algebraically closed field $K$ and any ring homomorphism $\varphi : A\rightarrow K$ such that the image of ${\rm Spec}(K)\rightarrow {\rm Spec}(A)$ is in $U$, we have

\begin{equation}
m_\alpha(\R Q)=m_\alpha(\mathcal{R}^\varphi Q),\hspace{.5cm}m_\alpha(\Pi_\R(Q))=m_\alpha(\Pi_{\mathcal{R}^\varphi}(Q)).
\label{m}\end{equation}

 If moreover $\R$ is local, 

\begin{equation}
a_\alpha(\R Q)=a_\alpha(\mathcal{R}^\varphi Q),\hspace{.5cm}a_\alpha(\Pi_\R(Q))=a_\alpha(\Pi_{\mathcal{R}^\varphi}(Q)).
\label{a}\end{equation}
\label{charchange}\end{theorem}

The rest of the section is devoted to the proof of  Theorem \ref{charchange}. We only prove it for the path algebra $\R Q$ as the proof for the preprojective algebra is similar.

If $X$ is a constructible subset of $\k^r$ defined by a finite number of equations $f_i=0$ and inequations $g_j\neq 0$, if $B\subset\k$ is a finitely generated $\Z$-algebra containing the coefficients of the polynomials $f_j$ and $g_j$ and if $\varphi:B\rightarrow K$ is a ring homomorphism into an algebraically closed  field $K$, then we denote by $f_i^\varphi, g_j^\varphi$ the corresponding polynomials with coefficients in $K$, and by $X^\varphi$ the constructible subset of $K^r$ defined by the equations $f_i^\varphi=0$ and inequations $g_j^\varphi\neq 0$.

Let us start with the following classical result (we give a proof for the convenience of the reader).

\begin{proposition}Let $X$ be a constructible subset of $\k^r$ defined by a finite number of polynomial equations $f_i=0$ and inequations $g_j\neq 0$. Let $B$ be a finitely generated $\Z$-subalgebra of $\k$ that contains the coefficients of the polynomials  $f_i$ and $g_j$. Then there exists a non-empty open subset $U$ of ${\rm Spec}(B)$, such that for any algebraically closed field $K$ and  homomorphism $\varphi:B\rightarrow K$ such that the image of ${\rm Spec}(K)\rightarrow{\rm Spec}(B)$ is in $U$, the constructible subset $X^\varphi$ of $K^r$  has the same dimension and the same number of irreducible components of maximal dimension as $X$.
\label{cons}\end{proposition}

\begin{proof}As a constructible subset of $\k^r$ is a finite union of locally closed subsets of $\k^r$ we may assume that $X$ is a locally closed subset of $\k^r$ which we regards as a $\k$-scheme. We put $S:={\rm Spec}(B)$ and we denote by $f_S:X_S\rightarrow S$ the map  that gives the structural map $f:X\rightarrow \k$ after extension of scalars from $B$ to $\k$. Let $\ell$ be a prime invertible in $\k$. Recall \cite[Proposition 2.5, Chapter 2]{deligne} that if $\calF$ is a $\overline{\Q}_\ell$-sheaf on a Noetherien scheme $Z$, there exists a finite partition of $Z$ in locally closed subschemes $Z_i$ such that $\calF|_{Z_i}$ is smooth. We apply this to the $\overline{\Q}_\ell$-sheaves $R^i(f_S)_!\overline{\Q}_\ell$ on $S$ (see \cite[\S 2.11, Chapter 2]{deligne}). Therefore there exists an open subset $U$ of $S$ such that for all $i$, $R^i(f_S)_!\overline{\Q}_\ell |_U$ is smooth. 

The fibers $(R^i (f_S)_!\mathcal{F})_{\overline{x}}\simeq H_c^i(X_S\times_SK,\overline{\Q}_\ell)$ over the geometric points $\overline{x}:{\rm Spec}(K)\rightarrow U$, with $K$ algebraically closed, are thus  (non-canonically) isomorphic.

Hence for the geometric points $\overline{s}_o:{\rm Spec}(\k)\rightarrow U\subset S$ (induced by the inclusion $B\subset \k$) and $\overline{s}:{\rm Spec}(\overline{\F}_q)\rightarrow U\subset S$  (above a closed point $s$ of $S$ with residue field $\F_q$) we have

$$
H_c^i(X,\overline{\Q}_\ell)\simeq H_c^i(X_S\times_S \overline{\F}_q,\overline{\Q}_\ell).
$$
If we denote by $\varphi:B\rightarrow \overline{\F}_q$ the map induced by $\overline{s}$, then $X_S\times_S\overline{\F}_q=X^\varphi$.  We conclude by recalling that the dimension of the top compactly supported $\ell$-adic cohomology equals the number of irreducible components of maximal dimension.

\end{proof}

For a Frobenius algebra $\R'$ over an algebraically closed field $\k'$ put
$$
X(\R'):=\{(x,g)\in \rep^\alpha(Q,\R')\times\G_\alpha(\R')\,|\, g\cdot x=x\}.
$$
If $\R'$ is local we denote by ${\rm Ind}^\alpha(Q,\R')$ the subset of $\rep^\alpha(Q,\R')$ of indecomposable representations. By Proposition \ref{indec}, it is a constructible set over $\k'$. The set

$$
X_{\rm ind}(\R'):=\{(x,g)\in {\rm Ind}^\alpha(Q,\R')\times\G_\alpha(\R')\,|\, g\cdot x=x\}
$$
is thus a constructible set over $\k'$. 

Let $A\subset \k$ be a finitely generated $\Z$-algebra as above, i.e. such that $\mathcal{R}^\varphi$ is a Frobenius $K$-algebra for all homomorphism $\varphi:A\rightarrow K$, with $K$ a field. Extending $A$ if necessary we may assume that both $X(\R)$ and $X_{\rm ind}(\R)$ are defined over $A$ (i.e. can be defined as the set of solutions in some $\k^r$ of equations and inequations with coefficients in $A$). Theorem \ref{charchange} is thus a consequence of Proposition \ref{cons} together with the following result.

\begin{proposition} For any algebraically closed field $K$ and any ring homomorphism $\varphi : A\rightarrow K$, we have $X(\R)^\varphi=X(\mathcal{R}^\varphi)$ and if $\R$ is local $X_{\rm ind}(\R)^\varphi=X_{\rm ind}(\mathcal{R}^\varphi)$.
\end{proposition}

\begin{proof}The proposition reduces to locally closed subsets of $\k^r$. Let $X=X/_A$ be an $A$-scheme of finite type. If $B$ is an $A$-algebra, we denote by $X/_B=X\times_A B$ the $B$-scheme obtained by scalar extension.  As a functor of points $X/_B$ is the composition of the functor $X/_A$ with the functor from the category of $B$-algebras to the category of $A$-algebras induced by the $A$-algebra structure on $B$, i.e. for any $B$-algebra $T$ we have

$$
X/_B(T)=X/_A(T).
$$

As $\mathcal{R}$ is free over $A$,  the functor

$$
B\mapsto X/_\mathcal{R}(\mathcal{R}\otimes_AB)=X/_A(\mathcal{R}\otimes_AB)
$$
from the category of $A$-algebras is representable (see \cite[Theorem 4]{Neron}) by an $A$-scheme $X_{\mathcal{R}/A}$ (called the Weil restriction of $X/_\mathcal{R}$ with respect to $A\subset\mathcal{R}$). We rephrase the proposition as follows : for any ring homomorphism $\varphi:A\rightarrow K$ into an algebraically closed field $K$ we have 

\begin{equation}
X_{\mathcal{R}/A}\times_A K= X_{\mathcal{R}^\varphi/K}
\label{equa}\end{equation}
where $X_{\mathcal{R}^\varphi/K}$ is the Weil restriction of the $\mathcal{R}^\varphi$-scheme $X/_{\mathcal{R}^\varphi}$ with respect to $K\subset\mathcal{R}^\varphi$. Indeed taking the $K$-points on both sides of (\ref{equa}) we get $X(R)^\varphi=X(\mathcal{R}^\varphi)$.

To see (\ref{equa}) we verify that the associated functors of points are the same. 

For a $K$-algebra $T$ we have

\begin{align*}
(X_{\mathcal{R}/A}\times_AK)(T)&= X_{\mathcal{R}/A}(T)\\&=X/_\mathcal{R}(\mathcal{R}\otimes_AT)\\
&=X/_A(\mathcal{R}\otimes_AT),
\end{align*}
while the functor $X_{\mathcal{R}^\varphi/K}$ takes the $K$-algebra $T$ to 

\begin{align*}
X_{\mathcal{R}^\varphi/K}(T)&=X/_K(\mathcal{R}^\varphi\otimes_KT)\\
&=X/_K(\mathcal{R}\otimes_AT)\\
&=X/_A(\mathcal{R}\otimes_AT).
\end{align*}

\end{proof}

\section{Finite locally free representation type}

In this section  we let $\k$ be an algebraically closed field, $\R$ a local Frobenius algebra over $\k$ and $Q=(I,E)$  a connected finite quiver. We are interested to find all examples such that
there are finitely many isomorphism classes of locally free indecomposable representations of $Q$ over $\R$. 

\begin{definition} The path algebra $\R Q$ has {\em finite locally free representation type} if there are finitely many isomorphism classes of locally free indecomposable representations of $\R Q$.
	\end{definition}

For $d\in \Z_{>0}$ we  denote by $\k_d$ the truncated polynomial ring $\k[t]/(t^d)$ in particular $\k_1\cong \k$ and  $\k_2\cong \k[\epsilon]$ by identifying $t$ with $\epsilon$.

\begin{theorem} \label{classif} $\R Q$ has finite locally free representation type if and only if \begin{itemize} \item $\R=\k$ and $Q$ is of $ADE$ Dynkin type
		\item $\R$ arbitrary and $Q$ is of type $A_1$
		\item $\R=\k_d$ for any $d>1$ and $Q$ is of type $A_2$
		\item $\R=\k_2$ or $\R=\k_3$  and $Q$ is of type $A_3$
		\item $\R=\k_2$ and $Q$ is of type $A_4$
		\end{itemize}
	\end{theorem}

The analogous statement where we dot not restrict to locally free modules can be found in Geiss-Leclerc-Schr\" oer \cite[\S 13.3]{GLS} (when $\R$ is a truncated polynomial ring).

\bigskip

We will need several results to prove Theorem~\ref{classif}.

\begin{proposition} \label{a2finite} Let $\R$ be a local commutative Frobenius algebra over the algebraically closed field $\k$. Let $Q$ be a quiver with underlying graph $A_2$.
	If $$\rep^{(1,1)}(Q,\R)/\G_{(1,1)}(\R)=\R/\R^\times$$  is finite of cardinality $d+1$ then $\R\cong \k_d=\k[t]/(t^d)$. 
\end{proposition}
\begin{proof}   A representation of the quiver $Q$ over $\R$ of rank $(1,1)$ is given by an element $a\in \R$. Two elements $a,b\in \R$ give the same representation up to isomorphism if $ae=b$ for some unit $e\in \R^\times=\GL_1(\R)$. In particular, in this case $(a)=(b)$. Conversely if $(a)=(b)$ with $a=rb$ and $b=sa$ then $a=ars$. If now $rs\in \m$ then $1-rs\notin \m$ thus $(1-rs)$ is a unit implying $a=0$. Then $(b)=0$ and so $b=0$ as well. If $rs\notin \m$ then it is a unit so both $r$ and $s$ are. Therefore $a=rb$ differ by the unit $r$.   Thus we have $\rep^{(1,1)}(Q,\R)/\G_{(1,1)}(\R)=\R/\R^\times$ counting the different principal ideals in $\R$. Assume $\R/\R^\times$ is finite.  Let $a,b\in \m$, where $\m\triangleleft \R$ is the maximal ideal.   As we have finitely many different principal ideals in $\R$ and $|\k|=\infty$ there will be $\lambda_1\neq \lambda_2\in \k$ such that $I:=(a+\lambda_1 b)=(a+\lambda_2 b)$. This implies that $b,a\in I$.  Thus any two proper
	principal ideals are contained in a third one. This means that there is a maximal proper principal ideal in $\R$. This ideal must contain all non-units and thus must equal ${\mathfrak m}$. If $(t)=\m$ then we see that every element in $\R$ can
	be written as $et^k$ for some $k\in \N$ and $e\in \R^\times$. As $\R$ is finite dimensional over $\k$ there must be a smallest $d\in \N$ such that $t^d=0$. Then the natural map $\k[t]/(t^d)\to \R$ is an isomorphism.
	Thus  $\R\cong \k_d=\k[t]/(t^d)$  is a $t$-adic ring and the classes in $\k_d/\k_d^\times$ have representatives $1,t,\dots,t^{d-1}$ and $0$. The result follows.
\end{proof}

\begin{proposition} \label{loop}  There are infinitely many non-isomorphic locally free indecomposable representations for a quiver whose underlying graph is not a tree over any local Frobenius $\k$-algebra $\R$. 
\end{proposition}
\begin{proof} This clearly holds over $\k$ as such a graph will contain a loop, and we know from Gabriel's theorem that a loop has infinitely many non-isomorphic indecomposable representations. Embedding $\k\subset \R$ as the subalgebra of constants, the same non-isomorphic indecomposable locally-free representations will have the same property over $\R$ as any isomorphism of quiver representations over $\R$ will induce one over the residue field $\R/\mathfrak{m}$ and $\k\subset \R\to \R/\mathfrak{m}$ maps isomorphically onto the residue field as $\k$ is algebraically closed. 
\end{proof}

\begin{proposition}\label{d4} There are infinitely many non-isomorphic locally free indecomposable representations for a quiver of type $D_4$ over $\k_d$ for $d>1$ and $\k$ algebraically closed. In particular, the same is true for any quiver with a vertex with at least three neighbours.
	\end{proposition} 

To prove Proposition~\ref{d4} we need the following

\begin{lemma} Let $d>1$, $A\in \GL_2(\k_d)$ and   $$e_\lambda=\left[\begin{array}{c}1 \\ \lambda t^{d-1} \end{array}\right]\in \k_d^2$$ for $\lambda\in \k$. Assume that $Ae_0=\mu_0 e_0$ and $Ae_1=\mu_1 e_1 $ for some $\mu_0,\mu_1\in \k_d^\times$. Then, for all $\lambda\in\k$, there exists $\mu_\lambda\in \k_d^\times$ such that $Ae_\lambda=\mu_\lambda e_\lambda $. 
\end{lemma}
\begin{proof} Let $A=[a_{ij}]_{i,j\in \{1,2\}}$. As $Ae_0=\mu_0 e_0$ we have $a_{21}=0$. Then $$Ae_\lambda= \left[\begin{array}{c} a_{11}+a_{12}\lambda t^{d-1}  \\ a_{22} \lambda t^{d-1}  \end{array}\right].$$ If $Ae_\lambda=\mu_\lambda e_\lambda$ then $$\mu_\lambda =a_{11}+\lambda a_{12}t^{d-1}$$ and $$a_{22}\lambda t^{d-1}=(a_{11}+a_{12}\lambda t^{d-1} ) \lambda t^{d-1}.$$ When $\lambda\neq 0$ this is equivalent with \beq{ide}a_{11}-a_{22} \in t\k_d.\eeq By assumption $\mu_1$ exists, this implies \eqref{ide}, and in turn gives the result. 
\end{proof}

\paragraph{\it Proof of Proposition~\ref{d4}.} By our main Theorem~\ref{main1}, we can assume that the arrows in the $D_4$ quiver are all oriented towards the central vertex labelled by $1$. Then a locally free representation of rank vector $(2,1,1,1)$ can be given by three vectors in $\k_d^2$. Let $\M_\lambda$ be the representation of this $D_4$ quiver given by  the three vectors $e_0,e_1,e_\lambda\in \k_d^2$. Clearly $M_\lambda$ is indecomposable.   Then $M_{\lambda_1}\cong M_{\lambda_2}$ implies the existence of an $A\in \GL_2(\k_d)$ taking $(e_0,e_1,e_{\lambda_1})$ to the triple $(e_0,e_1,e_{\lambda_2})$ up to scaling by invertible elements in $\k_d^\times$. By the above Lemma this implies that $e_{\lambda_1}$ and $e_{\lambda_2}$ differ by scaling with an invertible element in $k_d^\times$. This in turn implies $\lambda_1=\lambda_2$ showing that $M_\lambda$ is an infinite family of pairwise non-isomorphic indecomposable representations of our quiver. The result follows. 
\hfill\(\Box\) %\todo{Tamas: this needs to go to the end of the line somehow.} 

\paragraph{\it Proof of Theorem~\ref{classif}.} By Proposition~\ref{a2finite} if the Frobenius $\k$-algebra $\R$ is not isomorphic to $\k_d$ for some $d$, then for a quiver $Q$ of type $A_2$ we have infinitely many pairwise non-isomorphic  representations of rank vector $(1,1)$. The same holds for any quiver with a proper edge. When there are only edge loops in the quiver Proposition~\ref{loop} shows that there will be infinitely many pairwise non-isomorpic representations over any local Frobenius $\k$-algebra $\R$.

Thus we can assume that $\R\cong \k_d$ for some $d$. If a quiver $Q$ has finite locally free representation type over $\k_d$ then by Proposition~\ref{loop} it has to be a tree and in turn by Proposition~\ref{d4} it has to be linear, in other words of type $A_n$ for some $n\geq 0$.

  As $a_\alpha(\k_dQ)$ is independent of the orientation of the quiver by our main Theorem~\ref{main1} we can assume that our $Q$ is linearly oriented. Then \cite[Proposition 11.1]{GLS} implies\footnote{We thank Bernard Leclerc for pointing out to us this argument from \cite{GLS}.} that the number of isomorphism classes of non-injective locally free indecomposable representations of $\k_dQ$  agrees with the number of classes of all (i.e. not neccessarily locally-free) indecomposable representations of $\k_dQ_0$ where $Q_0$ is the subquiver of type $A_{n-1}$ of $Q$ leaving out the final vertex of $Q$. Then in turn \cite[Proposition 13.1]{GLS} finishes the proof of Theorem~\ref{classif}.\hfill\(\Box\)

Below we discuss the finite locally free representation type quivers of Theorem~\ref{classif} in more detail.

\subsection{$A_2$ over $\k_d$ }\label{a2kd} When the quiver $Q$ is of
type $A_2$, i.e. two vertices with one arrow, say, $1\to 2$ then a
representation of it over $\k_d$ of rank $(m,n)$ is given by an
$n\times m$ matrix over $\k_d$. Moreover two matrices correspond to
isomorphic representations if and only if they are equivalent using
elementary row and column operations. As $\k_d$ is a principal ideal
ring we can take the matrix to Smith normal form by \cite[Theorem
15.9]{brown}. This shows that the
indecomposable representations have rank vector $(1,1)$ , $(1,0)$ or
$(0,1)$. According to Proposition~\ref{a2finite} up to isomorphism the
rank $(1,1)$ indecomposable matrices are given by the $d$ matrices
$[t^i]$ for $i=0\dots d-1$ and there are the trivial zero
representations of rank $(1,0)$ and $(0,1)$. Thus we have
\begin{theorem} If the quiver $Q$ is of type $A_2$ then
$$
a_{\alpha}(\k_d Q)=\left\{ \begin{array}{cl} 1 & \alpha=(1,0) \mbox{
or } (0,1)\\ d & \alpha=(1,1) \\ 0 & \mbox{ otherwise }\end{array} \right.
$$
\end{theorem}
\begin{remark} We could have deduced the same result by
following the argument in the last paragraph of the proof of
Theorem~\ref{classif}. According to that the non-injective locally
free indecomposable representations of $Q$ are in bijection with the
indecomposable representations of the subquiver $Q_0$ which in this
case is just the $A_1$ quiver. There are $d$ indecomposable modules
over $\k_d$ and these correspond to the $d$ non-injective
indecomposable locally free representations of $Q$. One of them has
rank vector $(0,1)$ this corresponds to the free $\k_d$-module. The
remaining $d-1$ non-injective indecomposable locally free
representations are the ones $\k_d\stackrel{[t^i]}{\rightarrow}\k_d$
for $i=1\dots d-1$, which correspond to the $\k_d$-modules
$\coker([t^i])=\k_d/(t^i)$. (When $i=0$ the representation is
injective.)
\end{remark}

\subsection{$A_3$ over $\k_2$}

We can find all locally-free indecomposable representations of a
linearly oriented quiver $Q=1\rightarrow 2 \rightarrow 3$ of type
$A_3$ over $\k_2$ by following the argument in the last paragraph of
the proof of Theorem~\ref{classif}.  We have the three indecomposable
injective locally free representations of rank vector $(1,0,0)$,
$(1,1,0)$ and $(1,1,1)$, where all the maps are $0$ or
isomorphisms. The remaining ones are in bijection with the set of all
indecomposable representations of $Q_0$ (of type $A_2$) over
$\k_2$. We can read them off from the Auslander-Reiten quiver of
$\k_2Q_0$ in \cite[Figure 13.5.4]{GLS}.  We get $0\rightarrow \k_d$,
$\k_d\rightarrow 0$ and $\k_2/(t^i)\stackrel{[t^j]}{\rightarrow}
\k_2/(t^k)$ where $0\leq i,k<2$ and $2-k>j\geq i-k$. We have thus $9$
of such. Thus there must be $12=3+9$ indecomposable locally-free
representations of $Q$. It is easy to find them: we have the three simple
ones, four of the form
$\k_2\stackrel{[t^i]}{\rightarrow}\k_2\rightarrow 0$ and
$0\rightarrow\k_2\stackrel{[t^i]}{\rightarrow}\k_2$ for $0\leq i<2$,
four of the form
$\k_2\stackrel{[t^i]}{\rightarrow}\k_2\stackrel{[t^j]}{\rightarrow}\k_2$
with $0\leq i,j<2$ and one of the form
\beq{speci}\k_2\stackrel{\scriptsize \left[\begin{array}{c} 1 \\
0 \end{array} \right]
}{\longrightarrow}\k^2_2\stackrel{[t,1]}{\rightarrow}\k_2 \eeq
 
 \begin{theorem} \label{a3k2} When the quiver $Q$ is of type $A_3$ then we have $12$ locally free indecomposables over $\k_2$ as follows $$ a_{\alpha}(\k_2 Q)=\left\{ \begin{array}{cl} 1 & \alpha=(1,0,0), (0,1,0) \mbox{ or } (0,0,1)\\ 2 & \alpha=(1,1,0) \mbox{ or } (0,1,1)\\ 4 & \alpha=(1,1,1) \\ 1 & \alpha=(1,2,1)\\ 0 & \mbox{ otherwise }\end{array}   \right.$$ 
 \end{theorem}

\begin{remark} We can compare Theorem~\ref{a3k2} to the number of indecomposables $a_\alpha(\Pi_\k(Q))$ of the preprojective algebra $\Pi_{\k}(Q)$. First by our main Theorem~\ref{main2}  $a_\alpha(\Pi_\k(Q))=a_\alpha(\k_2(Q))$. Indeed we find an agreement when  comparing Theorem~\ref{a3k2} above with the dimension vectors of all indecomposables in the Auslander-Reiten quiver of $\Pi_\k(Q)$ in \cite[\S 20.1]{GLS1}.
\end{remark}

\begin{remark} Just as in \ref{a2kd} we can reformulate the problem of classifying the representations of the quiver $1\to 2 \leftarrow 3$  over $\k_2$ as the classification of an $n\times m$ matrix and an $n\times k$ matrix over $\k_2$ modulo simultaneous row operations and individual column operations. The indecomposables are identical as above (just reversing the second arrow) with the exception of the last one \eqref{speci} which needs to be replaced by
	$$\k_2\stackrel{\scriptsize \left[\begin{array}{c} 1 \\ 0 \end{array} \right] }{\longrightarrow}\k^2_2\stackrel{\scriptsize \left[\begin{array}{c} 1 \\ t \end{array} \right]}{\leftarrow}\k_2.$$
	\end{remark}

This way we get a normal form for the above classification problem, similar to the Smith normal form in the previous \S \ref{a2kd}. One can prove this directly with row and column operations, similarly to the proof of Smith normal form. We know from Theorem~\ref{classif} that there is no such normal form for the same problem over $\k_d$ for $d>3$, the case of $\k_3$ will be discussed in the next section.

\begin{remark} \label{a3orientation} It is interesting to compare the full representation theories of the two $A_3$ type quivers studied previously, namely the linearly oriented $Q=1\rightarrow 2 \rightarrow 3$ and the one in the previous section $Q^\prime = 1\to 2 \leftarrow 3$. They both have finite representation type by  \cite[Proposition 13.1]{GLS}.  Indeed $Q$ has $36$, and $Q^\prime$ has $42$ indecomposable representations \cite{GLS3}, out of which we have $12$ locally free in both cases. So there seems to be no direct relationship between the full representation theories of $Q$ and $Q^\prime$, only when one restricts to the locally free ones. 
	\end{remark}

\subsection{$A_3$ over $\k_3$}

To find all  indecomposable locally free representations of the linearly oriented $A_3$ type quiver $Q$ over $\k_3$ we can proceed the same way as above. First we find all indecomposable
representations of $Q_0$ of type $A_2$ over $\k_3$ by reading them off from the Auslander-Reiten quiver of $\k_3Q_0$ computed in \cite{GLS3}. There are $27$ of them. We can compute the non-injective, indecomposable and locally free  representations of $A_3$ over $\k_3$ corresponding to them, by applying the inverse Auslander-Reiten translation as in \cite[Proposition 11.1]{GLS}. After some computations and adding the three injective locally free representations we arrive at the following
\begin{theorem}  \label{a3k3} When the quiver $Q$ is of type $A_3$ then we have $30$ locally free indecomposable representations over $\k_3$ as follows $$ a_{\alpha}(\k_3 Q)=\left\{ \begin{array}{cl} 1 & \alpha=(1,0,0), (0,1,0) \mbox{ or } (0,0,1)\\ 3 & \alpha=(1,1,0) \mbox{ or } (0,1,1)\\ 9 & \alpha=(1,1,1) \\ 5 & \alpha=(1,2,1)\\ 2 & \alpha=(2,2,1) \mbox{ or } (1,2,2)\\ 1 & \alpha=(2,2,2) \\ 2 & \alpha=(2,3,2)\\ 0 & \mbox{ otherwise }\end{array}   \right.$$ 
\end{theorem}
\begin{remark} We can observe that the above multiset of rank vectors of indecomposable locally free representations over $\k_3$ of a quiver of type $A_3$ is symmetric under reflection through the middle vertex, e.g. the numbers for $(2,2,1)$ and $(1,2,2)$ agree. This symmetry follows from our main Theorem~\ref{main1} that the number of locally free indecomposable representations of a given rank vector is independent of the orientation of the quiver. However this symmetry is a surprise when one arrives at these rank vectors starting from all indecomposable representations of $Q_0$ over $\k_3$ as we did above. In particular, if we leave out the rank vectors of the three injective representations $(1,0,0)$, $(1,1,0)$ and $(1,1,1)$ the symmetry of the remaining multiset is lost.  
\end{remark}
\subsection{$A_4$ over $\k_2$}

One can find the locally free indecomposable representations of a linearly oriented quiver $Q$ of type $A_4$ over $\k_2$ by the method mentioned above (c.f. Remark~\ref{a3orientation}). Indeed a program of Crawley-Boevey \cite{crawley-boevey} has determined a list of $40$ such rank vectors of indecomposables as in the 
following Theorem~\ref{a4k2}. However in this case we could alternatively use $a_\alpha(\k_2 Q)=a_\alpha(\Pi_\k(Q))$ our second main Theorem~\ref{main2} and the computation of $a_\alpha(\Pi_\k(Q))$ in \cite[20.2]{GLS1} to deduce 
\begin{theorem} \label{a4k2} When the quiver $Q$ is of type $A_4$ then we have $40$ locally free indecomposable representations over $\k_2$ as follows $$ a_{\alpha}(\k_2 Q)=\left\{ \begin{array}{cl} 1 & \alpha=(1,0,0,0), (0,1,0,0),(0,0,1,0) \mbox{ or } (0,0,0,1)\\ 2 & \alpha=(1,1,0,0), (0,1,1,0) \mbox{ or } (0,0,1,1)\\ 4 & \alpha=(1,1,1,0) \mbox{ or } (0,1,1,1)\\ 8 & \alpha=(1,1,1,1)  \\ 1 & \alpha=(1,2,1,0) \mbox{ or }(0,1,2,1)\\ 2 & \alpha=(1,2,1,1) \mbox{ or } (1,1,2,1)\\ 6 & \alpha=(1,2,2,1) \\ 1 & \alpha=(1,2,2,2)\mbox{ or } (2,2,2,1) \\ 0 & \mbox{ otherwise }\end{array}   \right.$$ 
\end{theorem}

\begin{remark} Even though we found that a quiver of type $A_4$ over $\k_2$ has finite locally-free representation type with $40$ locally-free indecomposable representations, nothing like this holds for the full representation theory. Indeed, when the quiver is linearly oriented Skowro\'nski  proved \cite[Theorem]{skowronski} that the full representation type is tame. This means that there are infinitely many indecomposable representations, however they can be parametrized by up to one-dimensional families. Interestingly, when the orientation is not linear, the full representation type is wild, i.e. infinite and not tame. Thus while even the representation type of the full representation theory can change when one changes the orientation of the quiver, the subcategory of locally-free representations behaves much more uniformly.  
\label{counterexfull}\end{remark}

\begin{remark} The argument in the last paragraph of
the proof of Theorem~\ref{classif} can be combined with Skowro\'nski's result  \cite[Theorem]{skowronski} and Theorem~\ref{main1}  to find that the locally-free representation type of  a quiver of type $A_5$ over $\k_2$ is tame. By our Theorem~\ref{main2} this shows that the preprojective algebra of type $A_5$ has also tame representation type \footnote{Tameness in exact subcategories of module categories is not well established, here we simply mean that their constructible set of $1$-dimensional.}. Indeed this was the main studied example of the paper \cite[\S 14]{GLS1}, where they found that the dimension vectors of the indecomposable representations are given by a certain elliptic $E_8$ root system.  
\end{remark}

\section{Toric case}
\subsection{Indecomposable representations}
\label{toric-fmla}
Let $Q$ be a finite quiver. Let $V:=\{1,\ldots,n\}$ be the labeled set
of vertices and $E$ the set of edges of $Q$. For $d$ a positive
integer let $\k_d:=\F_q[t]/(t^d)$, where $\F_q$ is a finite field of
cardinality $q$. To alleviate the notation we will denote by $\calR_d$
instead of $\rep^1(Q,\k_d)$ the set of representations of $Q$ over
$\k_d$ of dimension $1$ at each vertex. We will call these
representations {\it toric}. For convenience we will use the notation
$A_d(Q,q)$ instead of $a_1(\k_dQ)$ for the number of isomorphism
classes of such representations that are (absolutely)
indecomposable, and we will use the notation $G_d=\k_d^\times\times\cdots\times \k_d^\times$ instead of $\G_1(\R)$. In what follows a {\it subgraph} of $Q$ will mean a
spanning subgraph; i.e., $\Gamma\subseteq Q$ with vertex set $V$ and
edge set $E(\Gamma)\subseteq E$.  Also, the orientation of $Q$ ends up
not playing a role in what follows and we will mostly ignore it.

 To any $\phi\in \calR_d$ we can attach
two combinatorial data 
$$
[\phi]:=(\Gamma,r),
$$
 where

i) $\Gamma \subseteq Q$ is the subgraph consisting of the same vertex
set $V$ with edge set $E(\Gamma)$  those $e\in E$ such
that $\phi(e)$ is non-zero.

ii) $r:E(\Gamma)\rightarrow \Z_{>0}$ is the function that to each edge
$e\in E(\Gamma)$ assigns the number $1\leq r(e) \leq d$ such that
$(\epsilon)^{d-r(e)}$ is the ideal generated by~$\phi(e)\in \k_d$ or,
equivalently, such that the annihilator of $\phi(e)$ in $\k_d$ is
$(t)^{r(e)}$.

 The first observation is the following.
\begin{lemma}
\label{indec-repn}
i) The representation $\phi$ is indecomposable if and only if $\Gamma$ is
connected.

ii) The representations
$\phi\in \calR_d$ with fixed
$[\phi]=(\Gamma,r)$ have the same stabilizer
$G_d(\Gamma,r)\leq G_d$.\end{lemma}
\begin{proof}
  The first statement i) is clear. To prove ii) note that 
  $G_d(\Gamma,r)$ consists of those $(u_1,\ldots,u_n)\in G_d$ such
  that for all $e\in E(\Gamma)$ we have
$$
u_iu_j^{-1}\phi(e)=\phi(e)
$$
where $e$
is an edge joining the vertices $i$
and $j$
(notation: $i\stackrel{e}{\rightarrow}j$).
This condition is equivalent to
\begin{equation}
\label{stab-conditions}
u_i-u_j\in (t)^{r(e)},
\end{equation}
 which only depends on $r(e)$ and
neither on the actual value $\phi(e)$ nor on the orientation of $e$.
\end{proof}
Let $\calR_d(\Gamma,r)\subseteq \calR_d$
be the subset of representations $\phi$
with $[\phi]=(\Gamma,r)$.
By Lemma~\ref{indec-repn} the number of orbits of $G_d$
acting on $\calR_d(\Gamma,r)$
equals $\#\calR_d(\Gamma,r)/[G_d:G_d(\Gamma,r)]$.  Then
\begin{equation}
\label{no-repn}
\#\calR_d(\Gamma,r)=(q-1)^{\#E(\Gamma)}q^{\sum_{e\in E(\Gamma)}(r(e)-1)},
\end{equation}
since
$$
(t)^{d-r}
=\{t^{d-r}(x_0+x_1t+\cdots+x_{r-1}t^{r-1})\, | \,
x_0\in \F_q^\times, x_1.\ldots,x_{r-1}\in\F_q\}.
$$

To compute $|G_d(\Gamma,r)|$ is a bit more tricky. Let
$(u_1,\ldots,u_n)\in G_d(\Gamma,r)$ and let $a_e:=u_i-u_j$ if
$i\stackrel{e}{\rightarrow} j$. Then $a_e\in (t)^{ r(e)}$,
independent of the orientation of $e$. However, the $a_e$'s are not
independent of each other as for a any cycle $c$ in $  \Gamma$
their sum over all edges of $c$ is zero.

We have
$$
a_e=\sum_{k= r(e)}^{d-1}a_{e,k}t^k, \qquad \qquad a_{e,k}\in \F_q.
$$
For fixed $k$ each cycle $c$ of $  \Gamma$ yields a linear equation
\begin{equation}
\label{a-eqn}
\sum_{e\in c, r(e)\leq k} a_{e,k}=0.
\end{equation}
Conversely, it is clear that any solution to all~\eqref{a-eqn} gives
rise to differences $a_e:=u_i-u_j$ for an element
$(u_1,\ldots,u_n)\in G_d(\Gamma,r)$.  Let $\tilde \delta_k(\Gamma,r)$ be the
dimension of the vector space of such solutions
of~\eqref{a-eqn}. Since $\Gamma$ is connected (by assumption) we find
$|G_d(\Gamma,r)|=q^{\tilde \delta(\Gamma,r)}|\k_d^\times|=q^{\tilde \delta(\Gamma,r)+d-1}(q-1)$,
where $\tilde \delta(\Gamma,r)=\sum_{k=1}^{d-1}\tilde \delta_k(\Gamma,r)$.

In particular, we have the following.
\begin{proposition}
There exists a polynomial in $\Z[T]$  which depends only on $d$ and the underlying graph of $Q$ (i.e. it does not depend on the orientation of $Q$),  such that for any finite field $\F_q$, the number $A_d(Q,q)$ is the evaluation at $q$ of that polynomial (for short we say that $A_d(Q,q)$ is a polynomial in $q$).
\end{proposition}

We can give a simple expression for $\tilde \delta(\Gamma,r)$
in terms of certain associated graphs.  The function $ r$
determines a filtration $\calE$
of length $d$ on the set of edges $E(\Gamma)$ of $ \Gamma$.
\begin{equation}
\label{filtr}
\calE: \quad \emptyset=E_0\subseteq E_1\subseteq \cdots \subseteq E_d=E(
\Gamma),
\end{equation}
where $E_k:=\{e \in E(\Gamma) \,|\,  r(e)\leq k\}$

For $1\leq k \leq d-1$ let $ \Gamma_k$ be the graph obtained from
$ \Gamma$ by contracting every edge in
$E( \Gamma)\setminus E_k= \{e \in E(\Gamma) \,|\, r(e)>k\}$.  Note
that all edges with $ r(e)=d$ get contracted for every $k$.  The edge
set of $ \Gamma_k$ is $E_k$. Then
$\tilde \delta_k(\Gamma,r)=\#E_k-b_1( \Gamma_k)$ where $b_1(\Gamma_k)$ is the first Betti number of the graph $\Gamma_k$. 
Together with~\eqref{no-repn} this gives
$\#\calR_d(\Gamma,r)/[G_d:G_d(\Gamma,r)]=(q-1)^{b_1(\Gamma)}q^{ \delta(\Gamma,r)}$, where 
$$
 \delta(\Gamma,r):=\sum_{k=1}^d(k-1)\#(E_k\setminus
  E_{k-1}) + \sum_{k=1}^{d-1}\left[\#E_k-b_1(
  \Gamma_k)-n+1\right].
$$
This quantity simplifies to

\begin{equation}
\label{m-defn}
\delta(\Gamma,r)=\sum_{k=1}^{d-1}[b_1( \Gamma)-b_1(  \Gamma_k)].
\end{equation}

In summary we have the following.
\begin{proposition}
\label{R-interpr}
The number of isomorphism classes of toric representations of $\Gamma$
where all maps between the $1$-dimensional spaces are non-zero scalars
has the form $(q-1)^{b_1(\Gamma)}R_d(\Gamma,q)$, where
\begin{equation}
\label{R-defn}
R_d(\Gamma,q):=\sum_rq^{\delta(\Gamma,r)}
\end{equation}
is a polynomial with non-negative integer coefficients, monic of
degree $(d-1)b_1(\Gamma)$.
\end{proposition}
\begin{proof}
To see that $R_d(Q,q)$ is monic of degree $(d-1)b_1(Q)$ note that the
  contributions in the sum on the right hand side of~\eqref{R-defn} of
  largest degree are those for which $r(e)=d$ for every edge.
\end{proof}

We have the following.
\begin{proposition}
 i) The polynomial $A_d(Q,q)$ is monic with integer coefficients and
  of degree $db_1(Q)$ and has the following expression
\begin{equation}
\label{A-fmla}
A_d(Q,q):=\sum_{(\Gamma,r)}(q-1)^{b_1(\Gamma)}q^{
\delta(\Gamma,r)},
\end{equation}
where the sum is over all pairs $(\Gamma,r)$ with $\Gamma\subseteq Q$
connected.  Equivalently,
\begin{equation}
\label{A-fmla-1}
A_d(Q,q)=\sum_\Gamma (q-1)^{b_1(\Gamma)}R_d(\Gamma,q),
\end{equation}
where the sum is over all connected subgraphs $\Gamma\subseteq Q$.

ii) Let $t(Q)$ be the number of spanning trees of $Q$ then
$$
A_d(Q,1)=d^{n-1}t(Q).
$$
(In both statements we disregard the orientation of $Q$.)
\end{proposition}
\begin{proof}
  i) Formulas~\eqref{A-fmla} and~\eqref{A-fmla-1} are immediate from
  the above discussion. To see that $A_d(Q,q)$ is monic of degree
  $db_1(Q)$ note that the contribution in the sum on the right hand
  side of~\eqref{A-fmla} of largest degree is for $\Gamma=Q$. The
  claim now follows from the computation of the degree of $R_d$ given
  in Proposition~\ref{R-interpr}.

ii) As for the value at $q=1$ the only contributions on the right hand
side of~\eqref{A-fmla} are when $\Gamma$ is a spanning tree. In that
case the corresponding term equals $1$ as
$b_1(\Gamma)=b_1(\Gamma_k)=0$ and there are $d^{n-1}$ choices for $r$.
This finishes the proof.
\end{proof}

\begin{remark} An equivalent formula to~\eqref{A-fmla} was proved by
  A.~Mellit (private communication), see also~\cite[Proposition 4.34]{Wyss}.
\end{remark}

As an example here is the explicit calculation of the number of
isomorphism classes of (absolutely) indecomposable representations of
dimension $(1^3)$ for the cyclic graph $C_3$ over $\k_2$ using the
above ideas.  We let $C_n$ be the graph on $n$ vertices, say
$1,2,\ldots,n$, with a single edge connecting $i$ to $i+1\bmod n$.
$$
\begin{array}{r|r|r|r|r|r|r}
m&r_{12}&r_{23}&r_{31}& |G_2(\Gamma,r)|& \#\calR_2(\Gamma,r) & \#{\rm orbits}\\
\hline
1&2&2&2&q(q-1)&q^3(q-1)^3&(q-1)q\\
3&2&2&1&q(q-1)&q^2(q-1)^3&q-1\\
3&2&1&1&q^2(q-1)&q(q-1)^3&q-1\\
1&1&1&1&q^3(q-1)&(q-1)^3&q-1\\
3&&2&2& q(q-1)&q^2(q-1)^2&1\\
6&&2&1&q^2(q-1)&q(q-1)^2&1\\
3&&1&1&q^3(q-1)&(q-1)^2&1
\end{array}.
$$

The table is a list of connected subgraphs $\Gamma \subseteq C_3$ and
the value $r_{ij}$ for each edge $ij$ of $\Gamma$ up to symmetry ($m$
denotes the number of such pairs). The other entries are:
$|G_2(\Gamma,r)|=$ the order of the stabilizer, $\#\calR_2(\Gamma,r)=$
the number of representations and
$\#\calR_2(\Gamma,r)/|G_2(\Gamma,r)|=$ the number of orbits
(isomorphism classes of representations) of the given type.  Adding
all the terms we have
$$
A_2(C_3,q)=q(q-1)+3(q-1)+3(q-1)+(q-1)+3+6+3=q^2+6q+5.
$$

In fact, it is not hard to do the calculation using~\eqref{A-fmla}
for $C_3$ and any depth $d$. We obtain the following.
$$
A_d(C_3,q)=q^d+6\sum_{k=1}^{d-1}(d-k)q^k+3d-1.
$$
We can extend this calculation to the general cyclic graph $C_n$.
\begin{proposition}
  For $n\geq 1$ let $C_n$ be the cyclic graph with vertices labelled
  $1,2,\ldots,n$ and an edge between the vertices $i$ and $i+1$ (read
  modulo $n$) for $i=1,\ldots,n$. Then for $d\geq 1$
\begin{equation}
\label{cyclic-fmla}
A_d(C_n,q)= q^d
+\sum_{k=1}^{d-1}\left[(d-k+1)^n-2(d-k)^n+(d-k-1)^n\right] q^k
+\left[-d^n+(d-1)^n+nd^{n-1}\right],
\end{equation}
a polynomial with non-negative integer coefficients. In particular,
\begin{equation}
\label{C2-fmla}
A_d(C_2,q)=q^d+2q^{d-1}+\cdots + 2q+1,
\end{equation}
and for the graph $C_1$ consisting of one vertex with a self-loop
$$
\label{C1-fmla}
A_d(C_1,q)=q^d.
$$
\end{proposition}
\begin{proof}
  We use~\eqref{A-fmla}. Consider first the case $\Gamma=C_n$. If $i$
  is the minimum value of $r$ then $b_1(\Gamma_k)=1$ for $k\geq i$ and
  $b_1(\Gamma_k)=0$ for $k <i$. Hence the pair $(\Gamma,r)$
  contributes $q^{i-1}$ to the sum. The total number of $r'$s with
  minimum value $i$ is $(d-i+1)^n-(d-i)^n$. Similarly, the total
  contributions of pairs $(\Gamma,r)$ for $\Gamma$ a segment obtained
  from $C_n$ by removing an edge is easily seen to be $nd^{n-1}$. A
  final manipulation of these terms yields~\eqref{cyclic-fmla}.  

To show that the  coefficients are non-negative note that the function
$f(x):= x^n$ is concave up for $n\geq 1$ and $x\geq 0$. Hence,
$(x+1)^n-2x^n+(x-1)^n\geq 0$ for $x\geq 0$. For the constant term,
by Taylor's theorem $f(d-1)=f(d)-f'(d)+\tfrac12f''(\theta)$, where $d-1\leq
\theta \leq d$. Since $f''(x)=n(n-1)x^{n-2}\geq 0$ for $x\geq0$ this
finishes the proof.
\end{proof}

\begin{remark}
i) Calculations such as these suggest the natural conjecture that in
general $A_d(Q,q)$ has non-negative integer coefficients. This is the
analogue for $\k_d$ of the first part of Kac's conjecture (now a
theorem~\cite{HLV}) on counting absolutely indecomposable
representations over~$\F_q$ (the case $d=1$).

ii) The polynomial $A_1(Q,q)$ is the specialization of the Tutte
polynomial $T(Q;x,y)$ of the underlying graph at $x=1$. Indeed, for
$d=1$ necessarily $r$ is identically equal to $1$ and hence
$\delta(\Gamma,r)=0$. Then~\eqref{A-fmla} yields
$$
A_1(Q,q)=\sum_\Gamma(q-1)^{b_1(\Gamma)}=T(Q;1,q),
$$
where $\Gamma$ runs over connected subgraphs of $Q$.  It is known that
the coefficients of the Tutte polynomial are non-negative integers and
therefore the same is true for $A_1(Q,q)$. For generalities on the
Tutte polynomial see for example~\cite{tutte-ref}.
\end{remark}

\subsection{Generating functions}
\label{toric-gen-fctns}
We now consider the generating function of both the $A_d$'s and the
$R_d$'s. It will be convenient to extend their definition to include a
term $A_0$ and $R_0$. Since we will need these independently later we
first define the following two functions on connected graphs.
\begin{definition}
For $\Gamma$ a connected graph  define
\begin{equation}
\label{epsilon-defn}
\epsilon(\Gamma):=
\begin{cases}
1 & \Gamma=\bullet\\
0 &\text{otherwise}
\end{cases}
\end{equation}
and
\begin{equation}
\label{epsilon1-defn}
\epsilon_1(\Gamma):=
\begin{cases}
1 & \Gamma=S_m, \text{ for some $m\geq 0$}\\
0 &\text{otherwise.}
\end{cases}
\end{equation}
where $\Gamma=S_m$ is the graph with one vertex and
$m$ loops. 

We then set 
$$
R_0(\Gamma,q):=\epsilon(\Gamma), \qquad \qquad 
A_0(\Gamma,q):=\epsilon_1(\Gamma).
$$
Note that with this definition the identity~\eqref{A-fmla-1} remains
valid for $d=0$.  We include also the case $m=0$ where $S_0=\bullet$, the
graph consisting of a single vertex with no edges.

\end{definition}

Let 
\begin{equation}
\label{Agf-defn}
A(Q,q,T):=\sum_{d\geq 0}A_d(Q,q)T^d, \qquad \qquad R(\Gamma,q,T):=\sum_{d\geq
  0}R_d(\Gamma,q)T^d. 
\end{equation}
We record for further reference that as it is easily seen
\begin{equation}
\label{R-1}
R(\bullet,q,T)=\frac1{1-T}.
\end{equation}

To study these series we introduce some notation. Fix a connected
subgraph of $\Gamma\subseteq Q$.  We would like to parametrize all
possible functions $r$ keeping track of the exponent of $q$
in~\eqref{A-fmla}. To this end, given the function~$r$ on edges,
let $1\leq r_1<r_2<\cdots < r_l\leq d$ be the distinct values it takes
on $E(\Gamma)$. We associate with $r$ a strict filtration
$\calF=\calF(r)$ of the edges of $\Gamma$ of length $l$
\begin{equation}
\label{F-filtr}
\calF: \quad \emptyset=F_0\subsetneq F_1\subsetneq \cdots \subsetneq F_l=E(
\Gamma).
\end{equation}
in such a way that $r$ is constant equal to $r_i$ on
$F_i\setminus F_{i-1}$ for $i=1,\ldots,l$.  Let $\Lambda_i$ be the
graph obtained by contracting all edges of
$E(\Gamma)\setminus F_{i-1}$ in $\Gamma$ for $i=1,\ldots,l$ and let
$c_i:=b_1(\Gamma)-b_1(\Lambda_i)$. For convenience we set $c_0:=0$. Note that
$c=c(\calF):=(c_0,c_1,\ldots,c_l)$ depends only on $\calF$ and not on $r$
itself.

Given a $l+1$-tuple of non-negative integers $c=(c_0,c_1,\ldots,c_l)$ define
the rational function
$$
R(c,q,T):=\frac{q^{\sum_{i=2}^lc_i} T^l}{(1-q^{c_0}T)(1-q^{c_1}T)\cdots
  (1-q^{c_l}T)}.
$$
If $l=1$ so $c=(0,c_1)$ this yields
$$
R(c,q,T)=\frac{T}{(1-T)(1-T^{c_1})}.
$$
We also extend this definition to the degenerate case where
$E(\Gamma)$ is empty. By our running assumption that $\Gamma$ is a
spanning subgraph of $Q$ it follows that in this case
$Q=\bullet$. Then $l=0,c=(0)$ and
$$
R(c,q,T)=\frac{1}{1-T}.
$$

We have the following.
\begin{proposition}
Fix the above notation and hypothesis. The generating
functions~\eqref{Agf-defn} can be expressed
as
\begin{equation}
\label{ratnl-fmla}
A(Q,q,T)=\sum_\Gamma (q-1)^{b_1(\Gamma)}R(\Gamma,q,T), \qquad
R(\Gamma,q,T)=\sum_\calF R(c(\calF),q,T),
\end{equation}
where $\Gamma$ runs over all connected subgraphs of $Q$ and $\calF$
over all strict filtrations of $E(\Gamma)$. In particular, $A(Q,q,T)$
and $R(\Gamma,q,T)$ are rational functions of $q$ and $T$.
\end{proposition}
\begin{proof}
It will be convenient to set $r_0:=1$ and define $s_i:=r_i-r_{i-1}$
for $i=1,\ldots, l$. Fix a filtration $\calF$ as in
~\eqref{F-filtr}. Then all $r$'s such that $\calF(r)=\calF$ can be
parametrized by integers $s_1,\ldots,s_l$ with $s_1\geq
0,s_2>0,\ldots, s_l>0$ and $s_1+\cdots+s_l\leq d-1$. The exponent of
$q$ in~\eqref{A-fmla} is then $\sum_{i=1}^l c_is_i.$
Summing over all such $s$ and $d$ we get
$$
\sum_{d\geq 1} \sum_s
 x_1^{s_1}\cdots x_l^{s_l}\ T^d = \frac{x_2\cdots x_l
  T^l}{(1-T)(1-x_1T)\cdots (1-x_lT)}.
$$
Now plugging in $x_i=q^{c_i}$ combined with~\eqref{A-fmla-1}
yield~\eqref{ratnl-fmla}.
\end{proof}

For example, for $Q=C_3$ spelling out~\eqref{ratnl-fmla} we find 
\begin{align}
\begin{split}
A(C_3,q,T)
=&(q-1)\left(\frac{6T^3}{(1-T)^3(1-qT)}
+\frac{6T^2}{(1-T)^2(1-qT)}
+\frac{T}{(1-T)(1-qT)}\right)
+\frac{6T^2}{(1-T)^3}
+\frac{3T}{(1-T)^2}\\
=&(q-1)\frac{T(T^2+4T+1)}{(1-T)^3(1-qT)}+\frac{3T(T+1)}{(1-T)^3}\\
=&\frac{(2q + 1)T^2 + (q + 2)T}{(1-T)^2(1-qT)}.
\end{split}
\end{align}
The first three terms in the sum correspond to strict filtrations of
$E(C_3)$. There are six filtrations with $c=(0,1,0,0)$, six with
$c=(0,1,0)$ and one with $c=(0,1)$. The last two terms correspond to the
three subgraphs of $C_3$ with two edges. Each subgraph has two filtrations
with $c=(0,0,0)$ and one with $c=(0,0)$. 

In general for the cyclic graph the expression~\eqref{cyclic-fmla}
gives the following for $n>1$
\begin{align}
\begin{split}
A(C_n,q,T)=&(q-1)\frac{TA_n(T)}{(1-T)^n(1-qT)}
+n\frac{TA_{n-1}(T)}{(1-T)^n}\\
=&\frac{T\sum_{j=0}^{n-2}A(n-1,j)\left[q(j+1)+n-1-j\right]T^j}
{(1-T)^{n-1}(1-qT)},   
\end{split}
\end{align}
where $A_n(T):=\sum_{j=0}^{n-1}A(n,j)T^j$ with $A(n,j)$ 
  the Eulerian numbers defined, for example, by the generating
  function
\begin{equation}
\label{eulerian}
\sum_{k\geq 0}(k+1)^nT^k=\frac{\sum_{j=0}^{n-1}A(n,j)T^j}{(1-T)^{n+1}}.
\end{equation}
Note the duality
\begin{equation}
\label{C_n-duality}
A(C_n,q^{-1},T^{-1})=(-1)^nA(C_n,q,T), \qquad n>1,
\end{equation}
which follows easily from the symmetry $A(n,j)=A(n,n-1-j)$ of the
Eulerian numbers. We will see below~\S\ref{duality} that this is a
general feature of such generating functions.

 Here are the first few cases of  $N_n$,
 the numerator of $A(C_n,q,T)$ divided by $T$
$$
\begin{array}{c|l}
n& N_n\\
\hline
2 &q + 1\\
3 &  (2q + 1)T + (q + 2)\\
4 & (3q + 1)T^2 + (8q + 8)T + (q + 3)\\
5 & (4q + 1)T^3 + (33q + 22)T^2 + (22q + 33)T + (q + 4)\\
6 & (5q + 1)T^4 + (104q + 52)T^3 + (198q + 198)T^2 + (52q +
104)T + (q + 5)\\
\end{array}
$$

\subsection{Recursion}
\label{toric-recursion}

\begin{proposition}
\label{R-recursion-prop}
The following recursion holds
\begin{equation}
\label{R-recursion}
R(\Gamma,q,T)=\epsilon(\Gamma)+T\sum_{A\subseteq E(\Gamma)} 
R(\Gamma/A,q,q^{b_1(\Gamma[A])}T),  
\end{equation}
where $\Gamma[A]$ denotes the subgraph of $\Gamma$ determined by the
subset of edges $A$ and $\Gamma/A$ the graph obtained by contracting
all edges in $A$. 
\end{proposition}
\begin{proof}
 For $\Gamma\neq \bullet$ the recursion follows from the definition by
 grouping in the sum all strict   filtrations with the same first term
 $\Lambda_1=\Gamma/A$.  For $\Gamma=\bullet$ it is a simple
 verification, namely
$$
1+\frac T {1-T}=\frac 1{1-T}.
$$
\end{proof}

As an application of this recursion, consider $\Gamma=S_m$ the graph
with one vertex and $m$ loops.  Formula~\eqref{R-recursion} in this
case yields
$$
R(S_m,q,T)=\frac T{(1-T)(1-q^mT)}+\sum_{i=1}^{m-1}\binom
mi\frac T{1-T}R(S_{m-i},q,q^iT), \qquad m\geq 1.
$$
For $m\geq 1$ the $R(S_m,q,T)$'s are rational functions of the form
$TF_m(q,T)/(T)_{m+1}$ for some polynomials $F_m$, where
$(T)_m:=\prod_{i=0}^{m-1}(1-q^iT)$.  Here are the first few values.
$$
\begin{array}{c|l}
m& F_m\\
\hline
1& 1\\
2&qT + 1\\
3&q^3T^2 + (2q^2 + 2q)T + 1\\
4& q^6T^3 + (3q^5 + 5q^4 + 3q^3)T^2 + (3q^3 + 5q^2 + 3q)T + 1.
\end{array}
$$
On the other hand, by~\eqref{R-interpr} we have
$$
R(S_m,q,T)=(q-1)^{-m}\sum_{d\geq 0}(q^d-1)^mT^d, \qquad m\geq 1.
$$
This shows that $F_m$ is a $q$-version of the Eulerian polynomial whose
coefficients are the Eulerian numbers $A(m,j)$
of~\eqref{eulerian}. These $q$-Eulerian polynomials were defined by
Carlitz~\cite{carlitz}; they have interesting combinatorial
interpretations. 

In fact, it is trivial to verify that $A_d(S_m,q)=q^{dm}$ for
$d\geq 1$ as the group $G_d$ acts trivially on representations in this
case. It is also true for $d=0$. Hence
$$
A(S_m,q,T)=\sum_{d\geq 0}q^{dm}T^d=\frac1{1-q^mT}
$$
and~\eqref{A-fmla-1} becomes
$$
\sum_{i=0}^m\binom m
i(q-1)^i\frac{F_i(q,T)}{(T)_{i+1}}=\frac{q^m}{1-q^mT}.
$$

\subsection{Duality}
\label{duality}
We now prove a remarkable duality property of the rational functions
$R(\Gamma,q,T)$ and $A(\Gamma,q,T)$.  They both are transformed in a
simple way under the inversion $(q,T)\mapsto (q^{-1},T^{-1})$.  Here is a
table of the values of $R(\Gamma,q,T)$ and $A(\Gamma,q,T)$ for a few
small graphs.

\begin{table}[H]
\setlength{\tabcolsep}{5mm} % separator between columns
\def\arraystretch{1.25} % vertical stretch factor
\centering
\begin{tabular}{|c|c|c|}
\hline
$\Gamma$ & $R$ & $A$\\
\hline
             \begin{tikzpicture}[anchor=base, baseline] 
\draw[fill=black] (0,0) circle (.05);
\addvmargin{1mm}
             \end{tikzpicture}
& $\frac 1{[0]}$ & $\frac1{[0]}$\\
\hline
             \begin{tikzpicture}[anchor=base, baseline] 
\draw[fill=black] (0,0) circle (.05);
\draw[thick] (0,0) arc [start angle=0, end angle=360, radius=0.3 cm];
\addvmargin{1mm}
             \end{tikzpicture}
& $\frac T{[01]}$ & $\frac 1{[1]}$\\
\hline
             \begin{tikzpicture}[anchor=base, baseline] 
\draw[fill=black] (0,0) circle (.05);
\draw[fill=black] (.5,0) circle (.05);
\draw[thick] (0,0) -- (.5,0);
\addvmargin{1mm}
             \end{tikzpicture}
& $\frac T{[0^2]}$ & $\frac T{[0^2]}$\\

\hline
             \begin{tikzpicture}[anchor=base, baseline] 
\draw[fill=black] (0,0) circle (.05);
\draw[fill=black] (.5,0) circle (.05);
\draw[thick] (0,0) to [bend right=45] (.5,0);
\draw[thick] (0,0) to [bend right=-45] (.5,0);
\addvmargin{1mm}
             \end{tikzpicture}
& $\frac{(T+1)T}{[0^21]}$ & $\frac {(q+1)T}{[01]}$\\

\hline
             \begin{tikzpicture}[anchor=base, baseline] 
\draw[fill=black] (0,0) circle (.05);
\draw[fill=black] (.5,0) circle (.05);
\draw[thick] (0,0) to [bend right=60] (.5,0);
\draw[thick] (0,0) -- (.5,0);
\draw[thick] (0,0) to [bend right=-60] (.5,0);
\addvmargin{1mm}
             \end{tikzpicture}
& $\frac{(qT^2+2qT+2T+1)T}{[0^212]}$ & $\frac {(q^2+q+1)T}{[02]}$\\

\hline
             \begin{tikzpicture}[anchor=base, baseline] 
\draw[fill=black] (0,0) circle (.05);
\draw[fill=black] (.5,.6) circle (.05);
\draw[fill=black] (1,0) circle (.05);

\draw[thick] (0,0) -- (1,0);
\draw[thick] (0.5,.6) -- (1,0);
\draw[thick] (0.5,.6) -- (0,0);

\addvmargin{1mm}
             \end{tikzpicture}
& $\frac{(T^2+4T+1)T}{[0^31]}$ & $\frac {(2qT+T+q+2)T}{[0^21]}$\\
\hline
\end{tabular}
\caption{Rational Functions\tablefootnote{$[0^{n_0}1^{n_1}\cdots]$
    denotes the polynomial $(1-T)^{n_0}(1-qT)^{n_1}\cdots$} }
\label{graph-table}
\end{table}

More precisely, we have the following
\begin{theorem}
\label{duality-thm}
  Let $\Gamma$ be a finite connected graph. Then we have
\begin{equation}
\label{A-duality}
A(\Gamma,q^{-1},T^{-1})=\epsilon_1(\Gamma)+(-1)^{\#V(\Gamma)}A(\Gamma,q,T),
\end{equation}
and
\begin{equation}
\label{R-duality}
R(\Gamma,q^{-1},T^{-1})=\epsilon(\Gamma)+(-1)^{\#E(\Gamma)-1}q^{b_1(\Gamma)}R(\Gamma,q,T),  
\end{equation}
where $V(\Gamma)$ and $E(\Gamma)$ denotes the set of vertices and
edges of $\Gamma$ respectively. 
\end{theorem}
Before we start with the proof properly it will be convenient to setup
the language of Hopf algebras we will use (see for
example~\cite{manchon} for an introduction). Let $\calH$ be the
commutative Hopf algebra over $\Q$ of linear combination of finite
graphs up to isomorphism and one-point join, self-loops and multiple
edges allowed (see~\cite[Example 23]{hopf-alg}). Concretely, $\calH$
is generated over $\Q$ by isomorphism classes of finite graphs modulo
the equivalence
$\Gamma_1 \coprod \Gamma_2\simeq \Gamma_1\times_v\Gamma_2$.  Here
$\Gamma_1 \coprod \Gamma_2$ denotes disjoint union and
$\Gamma_1\times_v\Gamma_2$ the one-point join of the two graphs at
vertices $v_i\in \Gamma_i$ for $i=1,2$; namely, the graph obtained by
identifying the two vertices $v_1$ and $v_2$ in
$\Gamma_1 \coprod \Gamma_2$. We will denote by
$\Gamma_1\simeq\Gamma_2$ the equivalence of two graphs $\Gamma_1$ and
$\Gamma_2$ if needed but we will mostly abuse notation and write the
graph $\Gamma$ for its equivalence class.

  The product in $\calH$
  is given by disjoint union and the coproduct by
$$
\Delta(\Gamma):=\sum_{A\subseteq E(\Gamma)} \Gamma[A] \otimes
\Gamma/A,
$$
where $\Gamma[A]$ and $\Gamma/A$ are as above.  The unit $1\in \calH$
is the class of the graph~$\bullet$ and the counit is the
map~$\epsilon$ induced from that previously defined
in~\eqref{epsilon-defn}. Namely,
$$
\epsilon(\Gamma):=
\begin{cases}
1&\text{if $\Gamma \simeq \bullet$}\\
0&\text{otherw.}
\end{cases}
$$

We can give $\calH$ a grading by setting
$\delta(\Gamma):=\#E(\Gamma)$. Then the degree zero piece is
$\Q\cdot 1$ and hence $\calH$ is connected.  Finally, the convolution
of two functions $f,g$ on graphs up to isomorphism and one-point join
with values in a commutative $\Q$-algebra $K$ is defined as
$$
(f*g)(\Gamma):=\sum_{A\subseteq E(\Gamma)} f(\Gamma[A])\cdot g(\Gamma/A).
$$
In what follows we will take $K:=\Q[q,q^{-1}]$.  Convolution gives an
associative product on $K$-valued functions on equivalence classes of
graphs. If $f$ and $g$ are multiplicative on disjoint unions of graphs
(giving rise to {\it characters} of the algebra structure of~$\calH$
by extending them by linearity) then the same holds for their
convolution~$f*g$. In fact, the characters of $\calH$ form a group
under convolution with identity
element~$\epsilon$~\cite[Prop. II.4.1]{manchon}. 

For example, the map on classes of graphs induced by the
function~$\epsilon_1$ defined in~\eqref{epsilon1-defn} is a character
and it is easily checked that if we convolve it with itself $k\geq 1$
times we get
$$
(\epsilon_1 * \cdots * \epsilon_1)(\Gamma) = 
\begin{cases}
(m+1)^{k-1} & \Gamma\simeq S_m, \text{ for some $m\geq 0$}\\
0 &\text{otherwise.}
\end{cases}
$$
The $R_d$'s defined in~\eqref{R-defn} when viewed as functions on graphs
also induce characters of $\calH$ as
$$
R_d(\Gamma,q)=R_d(\Gamma_1,q)R_d(\Gamma_2,q)
$$
if $\Gamma=\Gamma_1\times_v\Gamma_2$.
Finally, 
$$
\psi(q)(\Gamma):=q^{b_1(\Gamma)}, \qquad \qquad
(-1)^\delta(\Gamma):=(-1)^{\#E(\Gamma)}, 
$$
induce characters of $\calH$ as well. In all the above examples we will
denote the corresponding functions on classes of graph with the same
symbol and simply drop $\Gamma$ from the notation.

\begin{lemma}
\label{psi-inverse}
We have 
$$
(-1)^\delta\psi(q)*\psi(q)=\epsilon.
$$
In other words, the convolution inverse of $\psi(q)$ equals
$(-1)^\delta\psi(q)$.
\end{lemma}
\begin{proof}
By Lemma~\ref{b_1-char} below we have  
$$
\left((-1)^\delta\psi *\psi\right)(\Gamma)=
\sum_{A\subseteq E(\Gamma)} (-1)^{\#A}q^{b_1(\Gamma[A])}
q^{b_1(\Gamma/A)}=q^{b_1(\Gamma)}\sum_{A\subseteq E(\Gamma)}
(-1)^{\#A}=\epsilon(\Gamma)
$$
\end{proof}
We will need the following simple
\begin{lemma}
\label{b_1-char}
With the above notation
$$
b_1(\Gamma)=b_1(\Gamma[A])+b_1(\Gamma/A),\qquad
\delta(\Gamma)=\delta(\Gamma[A])+\delta(\Gamma/A)
$$
 for any subset $A\subseteq E(\Gamma)$) 
\end{lemma}

\begin{proof}{\it (of Theorem~\ref{duality-thm})}
%{\it Proof of Theorem~\ref{duality-thm}}
  It suffices to prove~\eqref{R-duality} as it is easily seen to
  imply~\eqref{A-duality}. With our Hopf algebra setup the
  recursion~\eqref{R-recursion} is equivalent to the following
\begin{equation}
\label{R_d-recursion}
R_{d+1}=\psi(q^d)*R_d, \qquad R_0=\epsilon, \qquad \qquad d\geq 0,
\end{equation}
and hence by induction
\begin{equation}
\label{R_d-recursion-1}
R_d(q) = \psi(q^{d-1})*\cdots *\psi(1), \qquad d\geq 1.
\end{equation}

Let $\tilde R_d(\Gamma,q)$ be the coefficient of $T^{-d}$ in the
expansion of $R(\Gamma,q,T)$ in powers of $T^{-1}$. Namely,
$$
R(\Gamma,q,T)=\sum_d \tilde R_d(\Gamma,q) T^{-d},
$$
where the sum is over integers $d\geq d_0$ for $d_0\in\Z$ (recall that
$R(\Gamma,q,T)$ is a rational function).  In fact we may take $d_0=1$
since $\tilde R_d$ is zero for $d\leq0$ by the following.
\begin{lemma}
\label{R-vanish}
The rational function $R(\Gamma,q,T)$ vanishes to order at least
one at  $T=\infty$.
\end{lemma}
\begin{proof}
  This follows immediately from the recursion~\eqref{R-recursion} by
  induction on the number of edges of~$\Gamma$.  Indeed, for~$\bullet$
  we have $R=(1-T)^{-1}$ (see~\ref{R-1}).  Now
  rewriting~\eqref{R-recursion} as
$$
R(\Gamma,q,T)=\frac{\epsilon(\Gamma)}{1-T}+\frac T{1-T}\sum_{\emptyset
  \subsetneq A\subseteq E(\Gamma)} R(\Gamma/A,q,q^{b_1(\Gamma[A])}T),
$$
it suffices to note that in the sum on the right hand side each
$\Gamma/A$ has fewer edges than $\Gamma$ and hence the claim follows.
\end{proof}

Expanding in powers of $T^{-1}$ the recursion~\eqref{R-recursion} we
obtain
\begin{equation}
\label{R-recursion-step}
\sum_{d\geq 1}\tilde R_d(\Gamma,q)T^{-d}=\epsilon(\Gamma)+
T\sum_{d\geq 1}
\sum_{A\subseteq E(\Gamma)} \tilde
R_d(\Gamma/A,q)\left(q^{b_1(\Gamma[A])}T\right)^{-d}.
\end{equation}
Equating constant coefficients on both sides
of~\eqref{R-recursion-step} we find that
$$
0=\epsilon +\psi(q^{-1})*\tilde R_1(q).
$$
Letting $\phi(q)$ be the convolution inverse of $\psi(q)$ and thanks
to Lemma~\ref{psi-inverse} we can rewrite this identity as
$$
\tilde R_1(q)=-\phi(q^{-1})
*\epsilon=-\phi(q^{-1})=(-1)^{\delta-1}\psi(q^{-1}) .
$$
In general, comparing the coefficients of $T^{-d+1}$  we find 
\begin{equation}
\label{R_d-tilde-recursion}
\tilde R_{d-1}(q)=\psi(q^{-d})*\tilde R_d(q), \qquad \qquad d> 1,
\end{equation}
which using $\phi$ again we may reformulate as
\begin{equation}
\label{R_d-tilde-recursion-1}
\tilde R_d(q)=(-1)^\delta\psi(q^{-d})*\tilde R_{d-1}(q), \qquad \qquad d> 1.
\end{equation}
Combined with Lemma~\ref{b_1-char} this gives
$$
\tilde R_d(q) =
(-1)^{\delta-1}\psi(q^{-1})\cdot\left(\psi(q^{-(d-1)})*\cdots
  *\psi(1)\right), \qquad d\geq 1.
$$
(To be clear, here $\cdot$ is just multiplication in $K$.)

Finally, comparison with~\eqref{R_d-recursion-1} gives
$$
\tilde R_d(q) =
(-1)^{\delta-1}\psi(q)\cdot R_d(q^{-1}), \qquad \qquad d> 0.
$$
proving~\eqref{R-duality}. 
\end{proof}

\begin{remark}
There is a relation between $R_2(\Gamma,q)$ and the Tutte polynomial
of $\Gamma$. More precisely, we have the following.
\begin{proposition}
For a connected graph $\Gamma$ 
\begin{equation}
\label{R2-fmla}
R_2(\Gamma,q)=T(\Gamma;2,q+1)
\end{equation}
\end{proposition}
\begin{proof}
  This is a simple verification using that $R_2(q)=\psi(q)*\psi(1)$
  by~\eqref{R_d-recursion-1} and the standard expression of the Tutte
  polynomial as a sum over $A\subseteq E(\Gamma)$
  (e.g. see~\cite[(3.3)]{tutte-ref}).
\end{proof}
The higher invariants $R_d(\Gamma,q)$ for $d>2$ are not expressible in
terms of the Tutte polynomial as they do not satisfy in general the
contraction/deletion property. Indeed, this fails already for
$
\Gamma=             \begin{tikzpicture}[anchor=base, baseline] 
\draw[fill=black] (0,0) circle (.05);
\draw[fill=black] (.5,0) circle (.05);
\draw[thick] (0,0) to [bend right=45] (.5,0);
\draw[thick] (0,0) to [bend right=-45] (.5,0);
\addvmargin{1mm}
             \end{tikzpicture}.
$
 Take $e$ to be any of the edges. Then
$
\Gamma \backslash e =
             \begin{tikzpicture}[anchor=base, baseline] 
\draw[fill=black] (0,0) circle (.05);
\draw[fill=black] (.5,0) circle (.05);
\draw[thick] (0,0) -- (.5,0);
\addvmargin{1mm}
             \end{tikzpicture}
$
and 
$\Gamma/e =
             \begin{tikzpicture}[anchor=base, baseline] 
\draw[fill=black] (0,0) circle (.05);
\draw[thick] (0,0) arc [start angle=0, end angle=360, radius=0.3 cm];
\addvmargin{1mm}
\end{tikzpicture}.  $ We compute the difference of their rational
functions $R$ (the individual values appear in
Table~\ref{graph-table}) and get
\begin{align}
\begin{split}
R(\Gamma,q,T)-R(\Gamma\backslash e,q,T)-R(\Gamma/e,q,T)&=\frac{((q +
  2)T^2 - 1)T}{(1-T)^3} \\
&= -T + (2q + 1)T^3 + (2q^2 + 4q + 2)T^4 + O(T^5).
\end{split}
\end{align}
As expected the coefficient of $T^2$ vanishes but all the higher ones
are non-zero. To check this we may set say $q=1$ and verify that the
coefficient of $T^n$ is then $n(n-2)$ for all $n\geq 1$.
\end{remark}


\begin{thebibliography}{}

\bibitem{Neron}{\sc Bosch S., L\"utkebohmert W., {\rm and} Raynaud
    M.}: {\em N\'eron Models},  Ergebnisse der Mathematik und ihrer
  Grenzgebiete (3) [Results in Mathematics and Related Areas (3)],
  {\bf 21}, Springer-Verlag, Berlin, 1990. x+325 pp. 

\bibitem{Bour}{\sc Bourbaki, N.}: {\em El\'ements of Mathematics}, Algebra I, Chap. II, Linear Algebra.

\bibitem{Brion}{\sc Brion, M.} : {\em Representations of quivers}. Geometric methods in representation theory. I, 103--144, S\'emin. Congr., 24-I, Soc. Math. France, Paris, 2012.

\bibitem{brown}{\sc Brown, W.}: {\em Matrices over commutative rings. }
Monographs and Textbooks in Pure and Applied Mathematics, {\bf 169}. {\em
  Marcel Dekker, Inc.,} New York, 1993.

\bibitem{carlitz} {\sc Carlitz, L.}
{\em A combinatorial property of q-Eulerian numbers},
Amer. Math. Monthly {\bf 82} (1975), 51--54. 

\bibitem{crawley-boevey}{\sc Crawley-Boevey, W.}: ({\em private communication})

\bibitem{deligne}{\sc Deligne, P.}: {\em Cohomologie \'etale}, SGA 4-1/2 IV, Lecture Notes in Mathematics, vol. {569}, Springer-Verlag, (1977), pp. 233--261.

\bibitem{Deligne-Weil}{\sc Deligne, P.}: {La conjecture de Weil : II}, Publications Math. IHES, {\bf 52} (1980), p. 137--252.

\bibitem{DG}{\sc Demazure M. {\rm and} Gabriel P.}: {\em Groupes alg\'ebriques},  North-Holland Publishing Co., Amsterdam, 1970. 

\bibitem{dlab-ringel}
{\sc V. Dlab {\rm and} C.~M. Ringel}, \emph{Indecomposable representations
	of graphs and algebras}, Mem. Amer. Math. Soc. \textbf{6} (1976), no.~173,
v+57.

\bibitem{eilenberg-nakayama}
{\sc Eilenberg S. {\rm and} Nakayama T.}:
{\em On the dimension of modules and algebras. II. Frobenius algebras and quasi-Frobenius rings}, 
Nagoya Math. J.{\bf 9} (1955), 1--16.

\bibitem{eisenbud} {\sc Eisenbud, D.}: {\em Commutative algebra.
	With a view toward algebraic geometry,} GTM {\bf 150}
      Springer-Verlag, New York, 1995.

    \bibitem{tutte-ref} {\sc Ellis-Monaghan, J. A {\rm and} Merino,
        C. }: {\em Graph Polynomials and Their Applications I: The
        Tutte Polynomial} Structural analysis of complex networks,
      pp. 219--255, Birkhäuser/Springer, New York, 2011.

\bibitem{gabriel1}
{\sc Gabriel P.}: \emph{Unzerlegbare {D}arstellungen. {I}}, Manuscripta Math.
\textbf{6} (1972), 71--103; correction, ibid. \textbf{6} (1972), 309.

\bibitem{gabriel2}{\sc Gabriel P.}: {\em
Indecomposable representations. II.} Symposia Mathematica, Vol. XI (Convegno di Algebra Commutativa, INDAM, Rome, 1971), pp. 81–104. Academic Press, London, 1973. 

\bibitem{geuenich} {\sc Geuenich J.}:{\em Quiver Modulations and Potentials}, PhD thesis, University of Bonn, 2017

\bibitem{GLS1} {\sc Geiss C., Leclerc B., {\rm and} Schr\"oer J.}:
{\em Semicanonical bases and preprojective algebras},
Ann. Sci. \'Ecole Norm. Sup. (4) {\bf 38} (2005), no. 2, 193--253. 

\bibitem{GLS}{\sc Geiss C., Leclerc B., {\rm and} Schr\"oer J.}: {\em Quivers with relations for symmetrizable matrices I: Foundations}, Invent. Math. {\bf 209} (2017), 61--158.

\bibitem{GLS3}{\sc Geiss C., Leclerc B., {\rm and} Schr\"oer J.}: ({\em private communication})

\bibitem{hazewinkel} { \sc Hazewinkel M., Gubareni Nn, {\rm and} Kirichenko V. V.}:
{\em Algebras, rings and modules Vol. 2}. 
Mathematics and Its Applications (Springer), 586.

\bibitem{hausel-kac} {\sc Hausel, T.}: {\em Kac conjecture from Nakajima quiver varieties},  Invent. Math. {\bf 181} (2010),  21--37.

\bibitem{aha}{\sc Hausel T., Letellier E., {\rm and} Rodriguez-Villegas F.}: {\em Arithmetic harmonic analysis on character and quiver varieties}, Duke Math. J. {\bf 160} (2011), no. 2, 323--400.

\bibitem{HLV}{\sc Hausel T., Letellier E., {\rm and}
    Rodriguez-Villegas F.}: {\em Positivity for Kac polynomials and
    DT-invariants of quivers},  Annals of Math. (2) {\bf 177} (2013),
  no. 3, 1147--1168.

\bibitem{HRV}{\sc Hausel, T, Rodriguez-Villegas, F.}:  {\em Mixed Hodge
    polynomials of character varieties}, With an appendix by Nicholas
  M. Katz. Invent. Math. {\bf 174} (2008), 555--624

\bibitem{Hua}{\sc Hua J.}: {\em Counting representations of quivers over finite fields}, J. Algebra {\bf 226} (2000), no. 2, 1011--1033.

\bibitem{jambor-plesken} {\sc Jambor S. {\rm and} Plesken W.}:
{\em Normal forms for matrices over uniserial rings of length two},
J. Algebra {\bf 358} (2012), 25--256. 

\bibitem{kac} {\sc Kac V.}: {\em Root systems, representations of
    quivers and invariant theory},  Invariant theory (Montecatini,
  1982), 74--108, {\em Lecture Notes in Mathematics}, {\bf 996},
  Springer Verlag 1983. 

\bibitem{hopf-alg}
{\sc Krajewski, T.,  Moffatt, I., Tanasa, A.}: {\em Hopf algebras and
  Tutte polynomials} Adv. in Appl. Math. {\bf 95} (2018), 271--330 


\bibitem{kraft-riedtmann}{\sc Kraft H., {\rm and} Riedtmann Ch.}: 
{\em Geometry of representations of quivers}, Representations of algebras (Durham, 1985), 109--145, 
London Math. Soc. Lecture Note Ser., 116, {\em Cambridge Univ. Press, Cambridge}, 1986. 

\bibitem{kaplan} {\sc Kaplan, D.}:
{\em Frobenius degenerations of preprojective algebras}.J. Noncommut. Geom.{\bf 14 }(2020), no.1, 34--411.

\bibitem{Let}{\sc Letellier, E.}: {\em DT-invariants of quivers and the Steinberg character of $\GL_n$}, IMRN {\bf 22} (2015), 11887--11908.

\bibitem{li-ye} {\sc Li, F. {\rm and} Ye, Ch.}:
	{\em Representations of Frobenius-type triangular matrix algebras}. Acta Math. Sin. {\bf 33}(2017), no.3, 341--361.

\bibitem{Macdonald}{\sc Macdonald I. G.}: {\em Symmetric Functions and Hall
Polynomials}, Oxford Mathematical Monographs, second ed., Oxford
Science Publications. The Clarendon Press Oxford University Press,
New York, 1995.

\bibitem{manchon} {\sc Manchon, D.}: {\em Hopf algebras, from basics
    to applications to renormalization} Comptes Rendus des Rencontres
  Mathematiques de Glanon 2001 (2003), arXiv:0408405v2 

\bibitem{milne}{\sc Milne, J. S.:}
{\em Algebraic groups. 
The theory of group schemes of finite type over a field. } Cambridge Studies in Advanced Mathematics, 170. {\em Cambridge University Press, Cambridge}, 2017
 
\bibitem{mozgovoy}{\sc Mozgovoy, S.}: {\em Motivic Donaldson-Thomas invariants  and McKay correspondence},  arXiv:1107.6044.


\bibitem{prasad} {\sc Prasad, A., Singla, P., Spallone, S.}:
{\em Similarity of matrices over local rings of length two},
Indiana Univ. Math. J. {\bf 64} (2015), 471--514

\bibitem{reiner}{\sc Reiner I.}:  {\em Maximal orders.} 1975  Oxford University Press
\href{http://www.ams.org/mathscinet-getitem?mr=1972204}{MR1972204}

\bibitem{ringel-zhang1} {\sc Ringel, C. M. {\rm and} Zhang, P.}:
{\em Representations of quivers over the algebra of dual numbers}, J. Algebra {\bf 475} (2017), 327--360.

\bibitem{ringel-zhang2}{\sc Ringel, C. M. {\rm and} Zhang, P.}:
{\em From submodule categories to preprojective algebras}, Math. Z. {\bf 278} (2014), no.1-2, 55--73.

\bibitem{schiffmann} {\sc Schiffmann, O.}:
  {\tt
    https://math.stackexchange.com/questions/606279/how-many-pairs-of}
{\tt
    -nilpotent-commuting-matrices-are-there-in-m-n-mathbbf-q}

 
\bibitem{skowronski}{\sc Skowro\'nski, A.}:
{\em Tame triangular matrix algebras over Nakayama algebras. }
J. London Math. Soc. (2){\bf 34} (1986), no. 2, 245--264. 

\bibitem{Springer}{\sc Springer, T. A.} : {\em Trigonometric sums, Green functions of finite groups and representations of Weyl groups}, Invent. Math. {\bf 36} (1976), 173--207.

\bibitem{stanley}{\sc Stanley R.}:
{\em Enumerative combinatorics}, Vol. 2. (English summary)
With a foreword by Gian-Carlo Rota and appendix 1 by Sergey Fomin. Cambridge Studies in Advanced Mathematics, 62. Cambridge University Press, Cambridge, 1999.

\bibitem{Wyss} {\sc Wyss, D.}: {\em Motivic and p-adic Localization
    Phenomena} {\tt arXiv:1709.06769v1}

\end{thebibliography}
\end{document}